\documentclass{amsart}
\usepackage[colorlinks = false,linktocpage=true]{hyperref}

\usepackage{amsmath,amssymb,amscd}
\usepackage{amsthm}
\numberwithin{equation}{section}
\AtBeginDocument{%
	\def\MR#1{}
}

\usepackage{color}
\usepackage{bookmark}

\usepackage{tikz}
\usetikzlibrary{positioning}
\usetikzlibrary{cd}
\usetikzlibrary{decorations.pathreplacing}
\usetikzlibrary{shapes,arrows}

\newtheorem{theorem}{Theorem}[section]
\newtheorem{proposition}[theorem]{Proposition}

\newtheorem{corollary}[theorem]{Corollary}
\newtheorem{lemma}[theorem]{Lemma}

\theoremstyle{definition}

\newtheorem{remark}[theorem]{Remark}
\newtheorem{example}[theorem]{Example}
\newtheorem{definition}[theorem]{Definition}

\DeclareMathOperator{\id}{id}

\DeclareMathOperator{\CC}{\mathbb{C}}
\DeclareMathOperator{\Q}{\mathbb{Q}}
\DeclareMathOperator{\R}{\mathbb{R}}

\DeclareMathOperator{\Z}{\mathbb{Z}}
\DeclareMathOperator{\Imag}{Im}
\DeclareMathOperator{\Li}{Li}

\newcommand*{\maxzero}[1]{[#1]_{+}}

\newcommand*{\dcn}{ h^{\vee}}
\newcommand*{\rf}{\mathcal{F}}
\newcommand*{\univsf}{\mathcal{F}_{\mathrm{sf}}}
\newcommand*{\tropsf}{\mathcal{F}_{\mathrm{trop}}}
\newcommand*{\fpspace}{\mathfrak{F}_{\gamma}}
\newcommand*{\gespace}{\mathfrak{D}_{\gamma}}

\newcommand{\tetvertex}{
	\begin{tikzpicture}
	\draw (0,0.4) rectangle (0.2,0.6);
	\end{tikzpicture}
}
\newcommand{\edgevertex}{
	\begin{tikzpicture}
	\draw [fill] (0,0.5) circle [radius=0.1] ;
	\end{tikzpicture}
}
\newcommand{\network}{\mathcal{N}}
\newcommand{\defnp}{
	\begin{tikzpicture}
	\draw (0,0.8) node{$(i,s)$};
	\draw [fill] (0,0.5) circle [radius=0.1] ;
	\draw [dashed] (0.2, 0.5) -- (1.2,0.5);
	\draw (1.4,0.8) node{$t$};
	\draw (1.3,0.4) rectangle (1.5,0.6);
	\end{tikzpicture}
}
\newcommand{\defn}{
	\begin{tikzpicture}
	\draw (0,0.8) node{$(i,s)$};
	\draw [fill] (0,0.5) circle [radius=0.1] ;
	\draw (0.7,0.7) node{$a$};
	\draw [<-] (0.2, 0.5) -- (1.2,0.5);
	\draw (1.4,0.8) node{$t$};
	\draw (1.3,0.4) rectangle (1.5,0.6);
	\end{tikzpicture}
}
\newcommand{\defnpp}{
	\begin{tikzpicture}
	\draw (0,0.8) node{$(i,s)$};
	\draw [fill] (0,0.5) circle [radius=0.1] ;
	\draw (0.7,0.7) node{$-a$};
	\draw [->] (0.2, 0.5) -- (1.2,0.5);
	\draw (1.4,0.8) node{$t$};
	\draw (1.3,0.4) rectangle (1.5,0.6);
	\end{tikzpicture}
}
\newcommand{\defNp}{
	\begin{tikzpicture}
	\draw (0,0.8) node{$e$};
	\draw [fill] (0,0.5) circle [radius=0.1] ;
	\draw [dashed] (0.2, 0.5) -- (1.2,0.5);
	\draw (1.4,0.8) node{$t$};
	\draw (1.3,0.4) rectangle (1.5,0.6);
	\end{tikzpicture}
}
\newcommand{\defN}{
	\begin{tikzpicture}
	\draw (0,0.8) node{$e$};
	\draw [fill] (0,0.5) circle [radius=0.1] ;
	\draw [<-] (0.2, 0.5) -- (1.2,0.5);
	\draw (1.4,0.8) node{$t$};
	\draw (1.3,0.4) rectangle (1.5,0.6);
	\end{tikzpicture}
}
\newcommand{\defNpp}{
	\begin{tikzpicture}
	\draw (0,0.8) node{$e$};
	\draw [fill] (0,0.5) circle [radius=0.1] ;
	\draw [->] (0.2, 0.5) -- (1.2,0.5);
	\draw (1.4,0.8) node{$t$};
	\draw (1.3,0.4) rectangle (1.5,0.6);
	\end{tikzpicture}
}

\begin{document}
\title{Jacobian matrices of Y-seed mutations}
\author[Y.~Mizuno]{Yuma Mizuno}
\address{Department of Mathematical and Computing Science,
	Tokyo Institute of Technology,
	2-12-1 Ookayama, Meguro-ku, Tokyo 152-8550, Japan.}
\email{mizuno.y.aj@m.titech.ac.jp}
\date{}
\begin{abstract}
For any quiver mutation sequence, we define a pair of matrices that describe a fixed point equation of a cluster transformation determined from the mutation sequence. We give an explicit relationship between this pair of matrices and the Jacobian matrix of the cluster transformation. Furthermore, we show that this relationship reduces to a relationship between the pair of matrices and the $C$-matrix of the cluster transformation in a certain limit of cluster variables.
As an application, we prove that quivers associated with once-punctured surfaces do not have maximal green or reddening sequences.
\end{abstract}
\maketitle
\tableofcontents
\section{Introduction}
Fomin and Zelevinsky introduced cluster algebras~\cite{FZ1}.
Subsequently, connections between cluster algebras and many areas of mathematics have been discovered.
\emph{Y-seed mutations} are one of the key ingredients in the theory of cluster algebras.
They first appeared in \cite{FZ1} and were studied in more detail in \cite{FZ4}.
They are discrete time evolution of a \emph{Y-seed} $(B,Y)$, where $B$  is a skew-symmetrizable matrix and $Y$ is a rational function.
They can be viewed as a generalization of Zamolodchikov's Y-systems, which were introduced in the study of the thermodynamic Bethe Ansatz~\cite{Zamolodchikov}.
A sequence of Y-seed mutations determines a rational function called \emph{cluster transformation}.
The purpose of this paper is to introduce a pair of matrices describing the fixed point equation of the cluster transformation and to understand the dynamics of the Y-seed mutations using this pair of matrices.

Set $n$ to be a fixed positive integer.
Let $B$ be an $n \times n$ skew-symmetrizable matrix and $y=(y_1 , \dots , y_n)$ be an $n$-tuple of commuting variables.
Given a sequence $m=(m_1,\dots , m_T)$ of integers in $\{ 1, \dots , n \}$ and a sequence $\sigma = (\sigma_1 ,\dots , \sigma_T)$ of permutations of $\{ 1, \dots , n \}$, we obtain the following transitions:
\begin{align}\label{intro:Y transition}
(B(0),Y(0)) \xrightarrow{\sigma_1 \circ \mu_{m_1}} (B(1),Y(1)) \xrightarrow{\sigma_2 \circ \mu_{m_2}} \cdots
\xrightarrow{\sigma_T \circ \mu_{m_T}} (B(T),Y(T)),
\end{align}
where $(B(0),Y(0)) = (B,y)$, $\mu_{m_t}$ is a Y-seed mutation (in the universal semifield) at the vertex $m_t$, and $\sigma_t$ acts on $B$ and $Y$ as $\sigma_t(B)_{ij} = B_{\sigma^{-1}(i)\sigma^{-1}(j)}$ and $\sigma(Y)_i = Y_{\sigma^{-1}(i)}$, respectively.
Each $Y(t)$ is an $n$-tuple of rational functions in $y$.
We say that the triple $\gamma = (B, m,\sigma)$ is a \emph{mutation sequence}.
We also define the (vector-valued) rational function $\mu_\gamma$ as $\mu_\gamma (y) = Y (T)$, and call it the cluster transformation of $\gamma$.

\begin{figure}
	\begin{tikzpicture}
	[baseline={([yshift=-.5ex]current bounding box.center)},scale=1.7,auto=left,black_vertex/.style={circle,draw,fill,scale=0.75},square_vertex/.style={rectangle,scale=0.8,draw}]
	\node (a) at (0,0) [black_vertex]{};
	\node (b) at (1,0) [black_vertex]{}; 
	\node (t1) at (0,-1) [square_vertex]{};
	\node (t2) at (1,-1) [square_vertex]{};
	\draw [dashed] (a)edge [bend left=15](t1);
	\draw [dashed] (t1)edge [bend left=15](a);
	\draw [dashed] (b)edge [bend left=15](t2);
	\draw [dashed] (t2)edge [bend left=15](b);
	\draw[arrows={-Stealth[scale=1.2]}] (b)--(t1);
	\draw[arrows={-Stealth[scale=1.2]}] (a)--(t2);
	\end{tikzpicture}\quad\quad
	\begin{tikzpicture}
	[baseline={([yshift=-.5ex]current bounding box.center)},scale=1.7,auto=left,black_vertex/.style={circle,draw,fill,scale=0.75},square_vertex/.style={rectangle,scale=0.8,draw}]
	\node (e1) at (90:1) [black_vertex]{};
	\node (t1) at (90-60:1) [square_vertex]{};
	\node (e3) at (90-60*2:1) [black_vertex]{};
	\node (t3) at (90-60*3:1) [square_vertex]{};
	\node (e2) at (90-60*4:1) [black_vertex]{};
	\node (t2) at (90-60*5:1) [square_vertex]{};
	\draw [dashed] (e1)edge [bend right=15](t1);
	\draw [dashed] (t1)edge [bend right=15](e3);
	\draw [dashed] (e3)edge [bend right=15](t3);
	\draw [dashed] (t3)edge [bend right=15](e2);
	\draw [dashed] (e2)edge [bend right=15](t2);
	\draw [dashed] (t2)edge [bend right=15](e1);
	\draw [arrows={-Stealth[scale=1.2] Stealth[scale=1.2]}] (t1)edge [bend right=15](e1);
	\draw [arrows={-Stealth[scale=1.2] Stealth[scale=1.2]}] (t1)edge [bend left=15](e3);
	\draw [arrows={-Stealth[scale=1.2] Stealth[scale=1.2]}] (t3)edge [bend right=15](e3);
	\draw [arrows={-Stealth[scale=1.2] Stealth[scale=1.2]}] (t3)edge [bend left=15](e2);
	\draw [arrows={-Stealth[scale=1.2] Stealth[scale=1.2]}] (t2)edge [bend right=15](e2);
	\draw [arrows={-Stealth[scale=1.2] Stealth[scale=1.2]}] (t2)edge [bend left=15](e1);
	\draw [arrows={-Stealth[scale=1.2] Stealth[scale=1.2]}] (e2)edge (t1);
	\draw [arrows={-Stealth[scale=1.2] Stealth[scale=1.2]}] (e3)edge (t2);
	\draw [arrows={-Stealth[scale=1.2] Stealth[scale=1.2]}] (e1)edge (t3);
	\end{tikzpicture}
	\caption{Examples of mutation networks. The left-hand mutation network appears in Example \ref{example: A2}, while the right-hand one appears in Example \ref{example: main theorem}.}
	\label{intro: mutation network}
\end{figure}
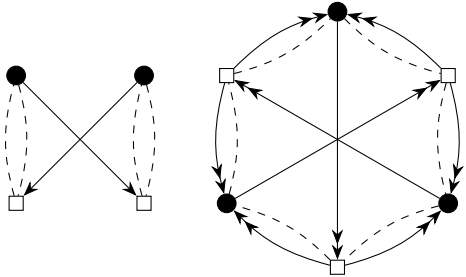

We can assign a combinatorial object $\network_\gamma$, called a \emph{mutation network}, to any mutation sequence $\gamma$.
Mutation networks were introduced in the physics literature~\cite{TerashimaYamazaki} to describe certain gauge theories.
They can be regarded as a generalization of ideal triangulations of 3-manifolds, which have been frequently used in the study of hyperbolic 3-manifolds since the work of Thurston~\cite{Thurston}.
Figure \ref{intro: mutation network} shows two examples of mutation networks.
From the combinatorial data of $\network_\gamma$, we define two integer matrices, $A_{+}$ and $A_{-}$,
which we call the \emph{Neumann-Zagier matrices} of $\network_\gamma$ (or of $\gamma$).
Given the relationship between mutation networks and ideal triangulations of 3-manifolds explained in~\cite{TerashimaYamazaki}, these matrices can be regarded as analogues of the Neumann-Zagier matrices for ideal triangulation of a 3-manifold, which were used by Neumann and Zagier to estimate the volume of a hyperbolic 3-manifold~\cite{NeumannZagier} (see Section \ref{section: surface examples}).
They are also related to the dilogarithm identities in conformal field theories when we consider mutation sequences on Dynkin quivers (see Section \ref{section: dynkin examples}).
The matrices $A_{\pm}$ have $T$ columns, equal to the length of the mutation sequence \eqref{intro:Y transition}.
We say that a mutation sequence is fully mutated if, roughly speaking, mutations occur at least once at all indices.
A more precise definition is provided in Section \ref{subsection: fully mutated}.
If a mutation sequence is fully mutated, the Neumann-Zagier matrices also have $T$ rows.
In this case, we can define a natural ordering of the rows, and $A_{+}$ and $A_{-}$ are $T \times T$ matrices.
In Section \ref{subsection: gluing equation}, we show that $A_{+}$ and $A_{-}$ actually describe the fixed point equation of the cluster transformation if $\gamma$ is fully mutated.

As mentioned before, our aim is to use the pair of matrices $(A_+,A_-)$ to understand the dynamics of the Y-seed mutations. 
The main result of this paper is to derive an explicit formula that gives a relationship between this pair of matrices and the Jacobian matrix of the cluster transformation.
Let $J_\gamma (y)$ be the Jacobian matrix of $\mu_\gamma$:
\begin{align*}
J_\gamma (y) = \left( \frac{\partial \mu_i}{\partial y_j} \right)_{1 \leq i,j \leq n},
\end{align*}
where $\mu_i$ is the $i$-th component of $\mu_\gamma(y)$.
Finally, define the $n \times n$ matrix $K_\gamma(y)$ as
\begin{align*}
K_\gamma(y) = 
\begin{pmatrix}
\mu_1(y) &  &    \\
 & \ddots  &   \\
  &  &  \mu_n(y)  
\end{pmatrix}^{-1}
J_\gamma (y)
\begin{pmatrix}
y_1 &  &     \\
 & \ddots  &  \\
 &   &  y_n  
\end{pmatrix}.
\end{align*}
For $t=1, \dots , T$, define rational functions $z_{t,+}$ and $z_{t,-}$ as
\begin{align*}
	z_{t,+} =  \frac{Y_{m_t}}{Y_{m_t}+1 } , \; 
	z_{t,-}  =  \frac{1}{Y_{m_t} +1 }    .
\end{align*}
Here, for brevity, we have simply written $Y_{m_t}(t-1)$ as $Y_{m_t}$.
Let $Z_{+}(y)$ and $Z_{-}(y)$ be the $T \times T$ diagonal matrices whose $t$-th entries are $z_{t,+}$ and $z_{t,-}$:
\begin{align*}
Z_{+}(y) = 
\begin{pmatrix}
z_{1,+} &  &     \\
& \ddots  &   \\
&  &  z_{T,+}  \\
\end{pmatrix},\;
Z_{-}(y) = 
\begin{pmatrix}
z_{1,-} &  &     \\
& \ddots  &   \\
&  &  z_{T,-}  \\
\end{pmatrix}.
\end{align*}

\begin{theorem}[Theorem \ref{theorem:det formula}]
	\label{theorem: intro det formula}
	Let $\gamma$ be a fully mutated mutation sequence.
	Then
	\begin{align}\label{intro: det formula}
	\det(I_n - K_\gamma(y)) = \det (A_+ Z_- (y) + A_-  Z_+ (y) ).
	\end{align}
\end{theorem}

An important feature of Theorem \ref{theorem: intro det formula}
is that the right-hand side is determined by a mutation sequence while the left-hand side depends only on the cluster transformation.
This means that the right-hand side is an ``invariant'' of a cluster transformation defined by data of mutation points.
For a fully mutated mutation sequence $\gamma$,
we denote the right-hand side of \eqref{intro: det formula} by $\tau_{\gamma} (y)$.
\begin{corollary}[Corollary \ref{corollary: tau = tau'}]
	\label{corollary: intro identity}
	Suppose that $\gamma$ and $\gamma'$ are fully mutated mutation sequences.
	Then $\mu_\gamma(y) = \mu_{\gamma'}(y)$ implies $\tau_{\gamma}(y) = \tau_{\gamma'}(y)$.
\end{corollary}

For example, let $B$ be the matrix defined by
\begin{align*}
B=
\begin{pmatrix}
	0 & -1 \\
	1 & 0
\end{pmatrix},
\end{align*}
and define two mutation sequences $\gamma$ and $\gamma'$ by
$\gamma = (B, (1,2) , (\id,\id))$ and $\gamma' =(B, (2,1,2), (\id, \id, (1 \; 2))$.
It is well-known that the cluster transformations of these two mutation sequences coincide, that is, $\mu_{\gamma}(y) = \mu_{\gamma'}(y)$.
Then Corollary \ref{corollary: intro identity}
yields the following identity:
\begin{align*}
	\det
	\begin{pmatrix}
	2 & -\frac{y_1}{1+y_1} \\
	-\frac{y_2(1+y_1)}{1+y_2+y_1 y_2} & 2
	\end{pmatrix} 
	= \det
	\begin{pmatrix}
	1 & -\frac{1 + y_2}{ 1+y_2 + y_1 y_2 } & 1\\
	1 & 1 & -\frac{1+y_2 + y_1 y_2}{(1+y_1)(1+y_2)} \\
	-\frac{1}{1+y_2} & 1 & 1
	\end{pmatrix},
\end{align*}
where the left-hand side is $\tau_{\gamma}(y)$ and the right-hand side is $\tau_{\gamma'}(y)$. See Example \ref{example: tau=tau'} for more detail.

As an application of Theorem \ref{theorem: intro det formula}, we prove that skew-symmetric matrices associated with once-punctured surfaces do not have maximal green or reddening sequences.
A \emph{maximal green} (and, more generally, \emph{reddening}) sequence is a special mutation sequence on a skew-symmetric matrix that was introduced by Keller~\cite{Keller2011}.
For a fixed skew-symmetric matrix, reddening sequences may or may not exist.
The existence of reddening sequences has important consequences. It produces explicit formulas of Kontsevich-Soibelman's refined Donaldson-Thomas invariants~\cite{Keller2011}, and also gives a sufficient condition for the existence of a theta basis of the upper cluster algebra~\cite{GHKK}.
(See~\cite{keller2019survey} for more details on maximal green and reddening sequences).
Thus it is important to determine whether a skew-symmetric matrix has reddening sequences or not.
The following was first proved by Ladkani~\cite{ladkani2013cluster} (for maximal green sequences),
and we give another proof for this result.

\begin{theorem}[Theorem \ref{theorem: no reddening}]
	\label{intro: no reddening}
	Let $B$ be a skew-symmetric matrix for an ideal triangulation of a once-punctured surface.
	Then $B$ does not have reddening or maximal green sequences.
\end{theorem}

Our proof is based on the relation between the Neumann-Zagier matrices and the \emph{$C$-matrix}, the matrix describing the \emph{tropical} Y-seed mutations, obtained by taking a $y \to 0$ limit in \eqref{intro: det formula} (see Theorem \ref{theorem: y0 limit}).

\subsection*{Acknowledgements}
The author gratefully acknowledges the help provided by Yuji Terashima.
The author also thanks Pavel Galashin, Shunsuke Kano, Akishi Kato, Atsuo Kuniba, Tomoki Nakanishi, and Masato Okado for their valuable comments.
This work is supported by JSPS KAKENHI Grant Number JP18J22576.

\section{Y-seed mutations}\label{section: y-seed mutation}
\subsection{Y-seeds and their mutations}
Here, we review Y-seed mutations in universal semifields, following~\cite{FZ4}.
First, set $n$ to be a fixed positive integer.
An $n \times n$ integer matrix $B=(B_{ij})_{1 \leq i,j \leq n}$ is skew-symmetrizable if 
there exist positive integers $d_1 , \dots , d_n$ such that $d_i B_{ij} = - d_j B_{ji}$.
Let $\rf$ be the field of rational functions in the variables $y_1 , \dots , y_n$ over $\Q$, and let
\begin{align}
\univsf = \left\{ \frac{f(y_1, \dots ,y_n)}{g(y_1 , \dots, y_n)} \in \mathcal{F} \;\middle|\; \parbox{16em}{$f$ and $g$ are non-zero polynomials in $\Q [y_1,\dots, y_n ] $ with non-negative coefficients.} \right\}
\end{align}
be the set of subtraction-free rational expressions in $y_1, \dots, y_n$ over $\Q$.
This is closed under the usual multiplication and addition.

A Y-seed is a pair $(B,Y)$ where $B=(B_{ij})_{1 \leq i,j \leq n}$ is an $n \times n$ skew-symmetrizable integer matrix, called an \emph{exchange matrix}, and $Y = (Y_1 \dots, Y_n)$ is an $n$-tuple of elements of $\univsf$.
Given this, Y-seed mutations are defined as follows.

\begin{definition}
	Let $(B,Y)$ be a Y-seed, and let $k \in \{ 1, \dots , n \}$.
	The Y-seed mutation $\mu_k$ is a transformation that transforms $(B,Y)$ into
	the Y-seed $\mu_k (B,Y) = (\widetilde{B} , \widetilde{Y} )$ defined as follows:
	\begin{align}
	\label{eq: B mutation}
	\widetilde{B}_{ij} =
	&\begin{cases}
	-B_{ij} & \text{if $i=k$ or $j=k$},\\
	B_{ij} + B_{ik}B_{kj} & \text{if $B_{ik}>0$ and $B_{kj}>0$}, \\
	B_{ij} - B_{ik}B_{kj}& \text{if $B_{ik}<0$ and $B_{kj}<0$}, \\
	B_{ij} & \text{otherwise.}
	\end{cases} \\
	\label{eq: Y mutation}
	\widetilde{Y}_i =
	&\begin{cases}
	Y_{k}^{-1} & \text{if $i=k$}, \\
	Y_{i}  (Y_k^{-1}+1)^{-B_{ki}} & \text{if $i \neq k$, $B_{ki} \geq 0$}, \\
	Y_{i}  (Y_k+1)^{-B_{ki}} & \text{if $i \neq k$, $B_{ki} \leq 0$}.
	\end{cases} 
	\end{align}
\end{definition}

When $B$ is skew-symmetric, it is convenient to identify it with a quiver.
A \emph{quiver} is a finite directed graph that may have multiple edges. Here, we assume that quivers do not have 1-loops and 2-cycles:
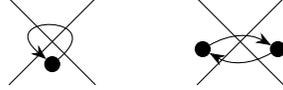
\begin{align*}
\begin{tikzpicture}
\node(a) at (0,0) [circle,draw,fill,scale=0.6]{};
\draw [arrows={-Stealth[scale=1.2]}] (a) edge [in=140,out=40,looseness=15] (a);
\node(b) at (2,0.2) [circle,draw,fill,scale=0.6]{};
\node(c) at (3,0.2) [circle,draw,fill,scale=0.6]{};
\draw[arrows={-Stealth[scale=1.2]}] (b) edge [bend left=30](c);
\draw[arrows={-Stealth[scale=1.2]}] (c) edge [bend left = 30](b);
\node (lu) at (-0.7,1.4-0.4) []{};
\node (rb) at (0.7,0-0.4) []{};
\node (lb) at (-0.7,-0.4) []{};
\node (ru) at (0.7,1.4-0.4) []{};
\draw[] (lu)--(rb);
\draw[] (lb)--(ru);
\node (lu2) at (-0.7+2.5,1.4-0.4) []{};
\node (rb2) at (0.7+2.5,0-0.4) []{};
\node (lb2) at (-0.7+2.5,-0.4) []{};
\node (ru2) at (0.7+2.5,1.4-0.4) []{};
\draw[] (lu2)--(rb2);
\draw[] (lb2)--(ru2);
\end{tikzpicture}.
\end{align*}
For a skew-symmetric integer matrix $B$ and indices $1 \leq i,j \leq n$,
we define the non-negative integer $Q_{ij}$ by
\begin{align}
	Q_{ij} = \maxzero{B_{ij}}
\end{align}
where $\maxzero{x} = \max(0,x)$.
Let $Q$ be a quiver with vertices labeled by the indices $1,\dots , n$ that has $Q_{ij}$ edges from $i$ to $j$ for each pair of vertices $(i,j)$.
The skew-symmetric matrix $B$ can then be recovered from $Q$ by
\begin{align}
	B_{ij} = Q_{ij} - Q_{ji}.
\end{align}
This correspondence gives a bijection between the set of $n \times n$ skew-symmetric integer matrices and the set of quivers with vertices labeled by $1, \dots, n$.
In the language of quivers, the mutation rules \eqref{eq: B mutation} can be described as follows:
\begin{enumerate}
	\item For each length two path $i \to k \to j$, add a new arrow $i \to j$.
	\item Reverse all arrows incident to the vertex $k$.
	\item Remove all 2-cycles.
\end{enumerate}
The transition
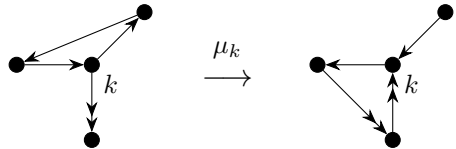
\[
\begin{tikzpicture}
[scale=1.0]
\node(k) at (0,0) [circle,scale=0.6,draw,fill]{};
\node(label) at (0.25,-0.25) []{$k$};
\node(a) at (-1,0) [circle,scale=0.6,draw,fill]{};
\node(b) at (0.7,0.7) [circle,scale=0.6,draw,fill]{};
\node(c) at (0,-1) [circle,scale=0.6,draw,fill]{};
\draw[arrows={-Stealth[scale=1.2] Stealth[scale=1.2]}] (k)--(c);
\foreach \from/\to in {k/b,a/k,b/a}
\draw[arrows={-Stealth[scale=1.2]}] (\from)--(\to);
\node(mutation) at (1.8,-0.2) [scale=1.2] {$\longrightarrow$};
\node(mutation) at (1.8,0.2) [] {$\mu_k$};
\node(k') at (4,0) [circle,scale=0.6,draw,fill]{};
\node(label') at (4.25,-0.25) []{$k$};
\node(a') at (3,0) [circle,scale=0.6,draw,fill]{};
\node(b') at (4.7,0.7) [circle,scale=0.6,draw,fill]{};
\node(c') at (4,-1) [circle,scale=0.6,draw,fill]{};
\draw[arrows={-Stealth[scale=1.2] Stealth[scale=1.2]}] (c')--(k');
\draw[arrows={-Stealth[scale=1.2] Stealth[scale=1.2]}] (a')--(c');
\draw[arrows={-Stealth[scale=1.2]}] (b')--(k');
\draw[arrows={-Stealth[scale=1.2]}] (k')--(a');
\end{tikzpicture}
\]
is an example of a quiver mutation, where we have omitted all labels other than the vertex $k$.

\subsection{Mutation sequences and mutation loops}
Let $\sigma$ be a permutation of $\{1 ,\dots , n \}$.
We define the action of $\sigma$ on a Y-seed by $\sigma(B,Y) = (\sigma(B),\sigma(Y) )$ where $\sigma(B)_{ij} = B_{\sigma^{-1}(i)\sigma^{-1}(j)}$ and $\sigma(Y)_i = Y_{\sigma^{-1}(i)}$.

Again, let $B$ be a skew-symmetrizable $n \times n$ matrix, and $y=(y_1 , \dots , y_n)$.
Let $m=(m_1,\dots , m_T)$ be a sequence of integers in $\{ 1, \dots , n \}$, and $\sigma = (\sigma_1 ,\dots , \sigma_T)$ be a sequence of permutations of $\{ 1, \dots , n \}$.
Let $(B(0),y(0)) = (B,y) $, and $(B(t),Y(t) ) = \sigma_t ( \mu_{m_t} (B(t-1),Y(t-1))$ for $t=1,\dots, T$.
This gives us the following Y-seed transitions:
\begin{align}\label{eq: Y transition}
(B(0),Y(0)) \xrightarrow{\sigma_1 \circ \mu_{m_1}} (B(1),Y(1)) \xrightarrow{\sigma_2 \circ \mu_{m_2}} \cdots
\xrightarrow{\sigma_T \circ \mu_{m_T}} (B(T),Y(T)).
\end{align}
We say that the triple $\gamma = (B,m,\sigma)$ is a \emph{mutation sequence}.
A mutation sequence is called a \emph{mutation loop} when $B(0)=B(T)$.
In addition, we call the (vector-valued) rational function $\mu_\gamma(y)$ given by
\begin{align}\label{eq: def of f}
	\mu_\gamma (y) = Y(T)
\end{align}
the \emph{cluster transformation of the mutation sequence $\gamma$}.
Iterating the transitions \eqref{eq: Y transition} only changes $Y$ if $B(0)=B(T)$, and these changes can be obtained by iteratively applying $\mu_\gamma$.
\begin{example}\label{example: mutation loop}
	If we define the skew-symmetric matrix $B$ by
	\begin{align*}
	B=
	\begin{pmatrix}
		0 & -1 & 2 & 2 & -1 \\
		1 & 0 & -3 & 0 & 2 \\
		-2 & 3 & 0 & -3 & 2 \\
		-2 & 0 & 3 & 0 & -1 \\
		1 & -2 & -2 & 1 & 0
	\end{pmatrix},
	\end{align*}
	the corresponding quiver $Q$ is given by
	\begin{align*}
	Q=
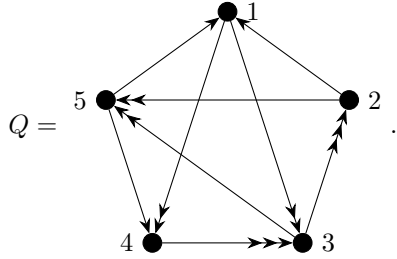
	\begin{tikzpicture}
	[baseline={([yshift=-.5ex]current bounding box.center)},scale=1.7,auto=left,black_vertex/.style={circle,draw,fill,scale=0.75}]
	\node (1) at (90:1) [black_vertex][label=right:${1}$]{};
	\node (2) at (90-72:1) [black_vertex][label=right:${2}$]{};
	\node (3) at (90-72*2:1) [black_vertex][label=right:${3}$]{};
	\node (4) at (90-72*3:1) [black_vertex][label=left:${4}$]{};
	\node (5) at (90-72*4:1) [black_vertex][label=left:${5}$]{};
	\draw[arrows={-Stealth[scale=1.2] Stealth[scale=1.2]}] (1)--(3);
	\draw[arrows={-Stealth[scale=1.2] Stealth[scale=1.2]}] (1)--(4);
	\draw[arrows={-Stealth[scale=1.2]}] (2)--(1);
	\draw[arrows={-Stealth[scale=1.2] Stealth[scale=1.2]}] (2)--(5);
	\draw[arrows={-Stealth[scale=1.2] Stealth[scale=1.2] Stealth[scale=1.2]}] (3)--(2);
	\draw[arrows={-Stealth[scale=1.2] Stealth[scale=1.2]}] (3)--(5);
	\draw[arrows={-Stealth[scale=1.2] Stealth[scale=1.2] Stealth[scale=1.2]}] (4)--(3);
	\draw[arrows={-Stealth[scale=1.2]}] (5)--(1);
	\draw[arrows={-Stealth[scale=1.2]}] (5)--(4);
	\end{tikzpicture}.
	\end{align*}
	Next, define the cyclic permutation $p$ by $p=(5\;4\;3\;2\;1)$.
	For a given non-negative integer $T$, define $m=(m_1, \dots , m_T)$ where $m_t = 1$ for all $t$, and $\sigma = (\sigma_1 , \dots, \sigma_T)$ where $\sigma_t = p$ for all $t$.
	Because
	\begin{align*}
	\mu_1  :
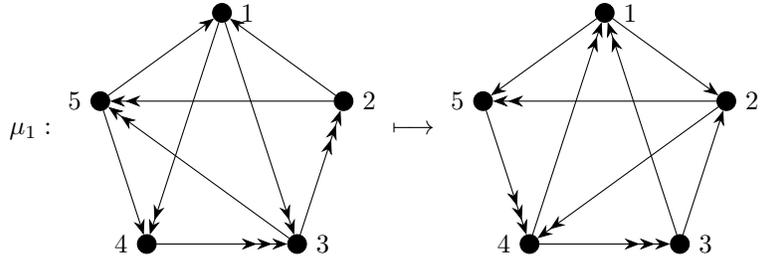
	\begin{tikzpicture}
	[baseline={([yshift=-.5ex]current bounding box.center)},scale=1.7,auto=left,black_vertex/.style={circle,draw,fill,scale=0.75}]
	\node (1) at (90:1) [black_vertex][label=right:${1}$]{};
	\node (2) at (90-72:1) [black_vertex][label=right:${2}$]{};
	\node (3) at (90-72*2:1) [black_vertex][label=right:${3}$]{};
	\node (4) at (90-72*3:1) [black_vertex][label=left:${4}$]{};
	\node (5) at (90-72*4:1) [black_vertex][label=left:${5}$]{};
	\draw[arrows={-Stealth[scale=1.2] Stealth[scale=1.2]}] (1)--(3);
	\draw[arrows={-Stealth[scale=1.2] Stealth[scale=1.2]}] (1)--(4);
	\draw[arrows={-Stealth[scale=1.2]}] (2)--(1);
	\draw[arrows={-Stealth[scale=1.2] Stealth[scale=1.2]}] (2)--(5);
	\draw[arrows={-Stealth[scale=1.2] Stealth[scale=1.2] Stealth[scale=1.2]}] (3)--(2);
	\draw[arrows={-Stealth[scale=1.2] Stealth[scale=1.2]}] (3)--(5);
	\draw[arrows={-Stealth[scale=1.2] Stealth[scale=1.2] Stealth[scale=1.2]}] (4)--(3);
	\draw[arrows={-Stealth[scale=1.2]}] (5)--(1);
	\draw[arrows={-Stealth[scale=1.2]}] (5)--(4);
	\end{tikzpicture}
	\longmapsto
	\begin{tikzpicture}
	[baseline={([yshift=-.5ex]current bounding box.center)},scale=1.7,auto=left,black_vertex/.style={circle,draw,fill,scale=0.75}]
	\node (1) at (90:1) [black_vertex][label=right:${1}$]{};
	\node (2) at (90-72:1) [black_vertex][label=right:${2}$]{};
	\node (3) at (90-72*2:1) [black_vertex][label=right:${3}$]{};
	\node (4) at (90-72*3:1) [black_vertex][label=left:${4}$]{};
	\node (5) at (90-72*4:1) [black_vertex][label=left:${5}$]{};
	\draw[arrows={-Stealth[scale=1.2] Stealth[scale=1.2]}] (2)--(4);
	\draw[arrows={-Stealth[scale=1.2] Stealth[scale=1.2]}] (2)--(5);
	\draw[arrows={-Stealth[scale=1.2]}] (3)--(2);
	\draw[arrows={-Stealth[scale=1.2] Stealth[scale=1.2]}] (3)--(1);
	\draw[arrows={-Stealth[scale=1.2] Stealth[scale=1.2] Stealth[scale=1.2]}] (4)--(3);
	\draw[arrows={-Stealth[scale=1.2] Stealth[scale=1.2]}] (4)--(1);
	\draw[arrows={-Stealth[scale=1.2] Stealth[scale=1.2] Stealth[scale=1.2]}] (5)--(4);
	\draw[arrows={-Stealth[scale=1.2]}] (1)--(2);
	\draw[arrows={-Stealth[scale=1.2]}] (1)--(5);
	\end{tikzpicture},
	\end{align*}
	the triple $\gamma_T := (B , m, \sigma)$ is a mutation loop for all $T$.
	This quiver $Q$ is an example of a period $1$ quiver in~\cite{FordyMarsh}.
	The cluster transformation for $\gamma_T$ is given by
	\begin{align*}
		\mu_{\gamma_T} (y_1 , y_2,y_3,y_4 , y_5) = \underbrace{\mu \circ \mu \circ \dots \circ \mu}_{T} (y_1 , y_2,y_3,y_4 , y_5) ,
	\end{align*}
	where $\mu$ is the cluster transformation of $\gamma_1$:
	\begin{align*}
		\mu (y_1 , y_2,y_3,y_4 , y_5)= 
		\left( 	y_2 (y_1+1) , \frac{ y_1^{2}  y_3 }{ (y_1+1)^{2} }  ,  
		\frac{ y_1^{2} y_4}{(y_1+1)^{2}} ,  y_5(y_1+1)  , \frac{1}{y_1} \right).
	\end{align*}
\end{example}

\section{Mutation networks}\label{section: mutation network}
\subsection{Definition of mutation networks}
\label{subsection: mutation network}
For any mutation sequence $\gamma$, we can construct a combinatorial object called mutation network, we will denote it by $\network_\gamma$.

Consider the sets defined by
\begin{align*}
	&\overline{\mathcal{E}}(t) = \{ (i,t) \mid 1\leq i \leq n  \}, \\
	&\overline{\mathcal{E}} = \bigcup_{0\leq t \leq T }\overline{\mathcal{E}}(t) . 
\end{align*}
We define the equivalence relation $\sim$ on $\overline{\mathcal{E}} $ as follows:
\begin{enumerate}
	\item For all $t=1 ,\dots ,T$, $(i,t-1) \sim (\sigma_t(i),t)$ if $i \neq m_t$.
	\item For all $i = 1 \dots, n$, $(i,T) \sim (i,0)$.
\end{enumerate}
Let $\mathcal{E} = \overline{\mathcal{E}} / \sim$ be the corresponding quotient set.

First, we define the graphs $\overline{\network}_{\gamma}(t)$ for all $t=1 ,\dots, T$.
These consist of two types of vertices: black vertices and a single square vertex.
The black vertices are labeled using elements of $\overline{\mathcal{E}}(t-1) \cup \overline{\mathcal{E}}(t)$, and represented as solid circles $\edgevertex$, while the square vertex is labeled $t$, and represented as a hollow square $\tetvertex$.
There are three types of edges joining these vertices:
broken edges, arrows from the square vertex to a black vertex, and arrows from a black vertex to the square vertex.
These edges are added as follows.
First, we add broken edges between the square vertex and black vertices labeled $(m_t,t-1)$ or $(\sigma_t(m_t),t)$.
Then, for all $i=1,\dots, n$, we add $B_{m_t,i}(t-1)$ arrows from the square vertex to the black vertex labeled $(i,t-1)$ if  $B_{m_t,i}(t-1)>0$.
Finally, for all $i=1,\dots, n$, we add $-B_{m_t,i}(t-1)$ arrows from the black vertex labeled $(i,t-1)$ to the square vertex if  $B_{m_t,i}(t-1)<0$.
These rules can be summarized as follows:
	\begin{enumerate}
		\item $\defnp$ if $(i,s) = (m_t,t-1)$ or $(\sigma_t(m_t),t)$,
		\item $\defn$ if $a=B_{m_t,i}(t-1)>0$ and $s = t-1$, for all $i$,
		\item $\defnpp$ if $a=B_{m_t,i}(t-1)<0$ and $s = t-1$, for all $i$.
	\end{enumerate}

Let $ \overline{\network}_\gamma = \cup_{1 \leq t \leq T} \overline{\network}_\gamma (t)$.
Then the mutation network $\network_{\gamma}$ of the mutation sequence $\gamma$ is the quotient of $\overline{\network}_\gamma$ by $\sim$:
\begin{align}
	\network_{\gamma} = \overline{\network}_\gamma / \sim.
\end{align}
Here, the quotient carried out such that the vertices are labeled by elements of $\mathcal{E}$, and all edges are included (i.e., we do not cancel the arrows pointing in the opposite direction).

\begin{example}\label{example: A2}
	Define the skew-symmetric matrix $B$ by
	\[
		B=
		\begin{pmatrix}
		0 & -1 \\
		1 & 0  
		\end{pmatrix}.
	\]
	Let $\gamma = (B, m, \sigma)$ be a mutation loop, where $m=(1,2)$ and $\sigma = (\id , \id)$.
	If we identify skew-symmetric matrices with quivers, the quiver transitions are given by
	\[
	\begin{tikzpicture}
	[scale=0.9]
	\node (1) at (0,0) [] {\underline{1}};
	\node (2) at (1,0) [] {2};
	\draw[arrows={-Stealth[scale=1.1]}] (2)--(1);
	\node (mu1) at (2.5,-0.1) [] {$\longrightarrow$};
	\node (mu1') at (2.5,0.2) [] {$\mu_1$};
	\node (1') at (4,0) [] {1};
	\node (2') at (5,0) [] {\underline{2}};
	\draw[arrows={-Stealth[scale=1.1]}] (1')--(2');
	\node (mu1) at (6.5,-0.1) [] {$\longrightarrow$};
	\node (mu1') at (6.5,0.2) [] {$\mu_2$};
	\node (1'') at (8,0) [] {1};
	\node (2'') at (9,0) [] {2};
	\draw[arrows={-Stealth[scale=1.1]}] (2'')--(1'');
	\end{tikzpicture},
	\]
	where the underlines indicate the mutated vertices.
	The graphs $\overline{\network}_\gamma (t)$ can then be given as
	\begin{align*}
	\overline{\network}_{\gamma}(1)=
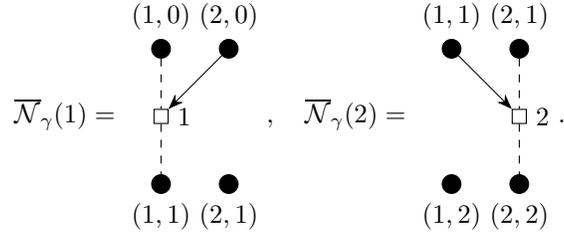
	\begin{tikzpicture}
	[baseline={([yshift=-.5ex]current bounding box.center)},scale=0.9,auto=left,black_vertex/.style={circle,draw,fill,scale=0.75},square_vertex/.style={rectangle,scale=0.8,draw}]
	\node (a) at (0,0) [black_vertex][label=above:${(1,0)}$]{};
	\node (t1) at (0,-1) [square_vertex][label=right:$1$]{};
	\node (ap) at (0,-2) [black_vertex][label=below:${(1,1)}$]{};
	\node (bp) at (1,-2) [black_vertex][label=below:${(2,1)}$]{};
	\node (b) at (1,0) [black_vertex][label=above:${(2,0)}$]{}; 
	\draw [dashed] (a)--(t1);
	\draw [dashed] (t1)--(ap);
	\draw[arrows={-Stealth[scale=1.2]}] (b)--(t1);
	\end{tikzpicture}, \quad
	\overline{\network}_{\gamma}(2) = 
	\begin{tikzpicture}
	[baseline={([yshift=-.5ex]current bounding box.center)},scale=0.9,auto=left,black_vertex/.style={circle,draw,fill,scale=0.75},square_vertex/.style={rectangle,scale=0.8,draw}]
	\node (ap) at (-1,0) [black_vertex][label=above:${(1,1)}$]{};
	\node (b) at (0,0) [black_vertex][label=above:${(2,1)}$]{}; 
	\node (a) at (-1,-2) [black_vertex][label=below:${(1,2)}$]{}; 
	\node (bp) at (0,-2) [black_vertex][label=below:${(2,2)}$]{}; 
	\node (t2) at (0,-1) [square_vertex][label=right:$2$]{};
	\draw [dashed] (b)--(t2);
	\draw [dashed] (t2)--(bp);
	\draw[arrows={-Stealth[scale=1.2]}] (ap)--(t2);
	\end{tikzpicture}.
	\end{align*}
	Let $\mathcal{E} = \{ e_1 , e_2 ,e_3 \} $
	be the set of black vertices in $\network_\gamma$,
	where
	\begin{align*}
	&e_1 = \{ (1,0), (1,1) ,(1,2) \}, \\
	&e_2 = \{ (2,0), (2,1) ,(2,2) \}.
	\end{align*}
	The mutation network of $\gamma$ is then given by
	\begin{align}\label{eq: A2 network}
	\network_{\gamma} = 
	\;
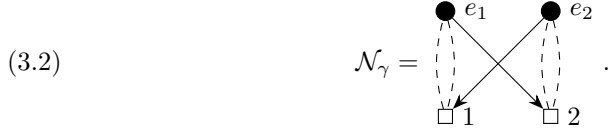
	\begin{tikzpicture}
	[baseline={([yshift=-.5ex]current bounding box.center)},scale=1.4,auto=left,black_vertex/.style={circle,draw,fill,scale=0.75},square_vertex/.style={rectangle,scale=0.8,draw}]
	\node (a) at (0,0) [black_vertex][label=right:$e_1$]{};
	\node (b) at (1,0) [black_vertex][label=right:$e_2$]{};
	\node (t1) at (0,-1) [square_vertex][label=right:$1$]{};
	\node (t2) at (1,-1) [square_vertex][label=right:$2$]{};
	\draw [dashed] (a)edge [bend left=15](t1);
	\draw [dashed] (t1)edge [bend left=15](a);
	\draw [dashed] (b)edge [bend left=15](t2);
	\draw [dashed] (t2)edge [bend left=15](b);
	\draw[arrows={-Stealth[scale=1.2]}] (b)--(t1);
	\draw[arrows={-Stealth[scale=1.2]}] (a)--(t2);
	\end{tikzpicture}.
	\end{align}
\end{example}

\subsection{Neumann-Zagier matrices of a mutation network}
\begin{definition}
	Let $\gamma$ be a mutation sequence and $\network_\gamma$ be its mutation network.
	For any given pair consisting of a black vertex labeled $e \in \mathcal{E}$ and a square vertex labeled $t \in \{ 1, \dots , T  \}$, 
	let $n_{et}^{0}$ be the number of broken lines between $e$ and $t$,
	$n_{et}^{+}$ be the number of arrows from $t$ to $e$,
	and $n_{et}^{-}$ be the number of arrows from $e$ to $t$:
	\begin{align}
	n_{et}^{0} &= \# \left\{ \defNp  \right\}, \\
	n_{et}^{+} &= \# \left\{ \defN  \right\}, \\
	n_{et}^{-} &= \# \left\{ \defNpp  \right\}.
	\end{align}
	Let $N_0$ be the $\mathcal{E} \times T$ matrix whose $(e,t)$-entry is $n_{et}^{0} $.
	Define the $\mathcal{E} \times T$ matrices $N_+$ and $N_-$ likewise.
	We say that the matrices $N_0$, $N_+$, and $N_-$ are the \emph{adjacency matrices} of the mutation network $\network_\gamma$.
\end{definition}

\begin{definition}
	Let $N_0$, $N_+$, and $N_-$ be the adjacency matrices of the mutation network $\network_\gamma$.
	Define two additional $\mathcal{E} \times T$ matrices $A_+ =(a_{et}^+)$ and $A_-  =(a_{et}^-)$ by
	\begin{align}
	\label{eq: def of A+}
		&A_+ = N_0 - N_+ , \\
	\label{eq: def of A-}
		&A_- = N_0 - N_- .
	\end{align}
	We call the matrices $A_+$ and $A_-$ the \emph{Neumann-Zagier matrices} of the mutation sequence $\gamma$ (or of the mutation sequence $\gamma$).
	We also define $A_0 =(a_{et}^0)$ by $A_0 = N_0$.
	For a given sequence $\boldsymbol{\varepsilon} = (\varepsilon_1 , \dots, \varepsilon_T) \in \{ \pm,0 \}^T$, let $A_{\boldsymbol{\varepsilon}}$ be the $\mathcal{E} \times T$ matrix whose $t$-th column is the $t$-th column in $A_{\varepsilon_t}$.
	We call $A_{\boldsymbol{\varepsilon}}$ the \emph{$\boldsymbol{\varepsilon}$-signed Neumann-Zagier matrix} of the mutation sequence $\gamma$.
\end{definition}

\begin{example}\label{example: A2 matrix}
	Let $\gamma$ be the mutation loop defined in Example \ref{example: A2}.
	It follows from $\eqref{eq: A2 network}$ that the adjacency matrices of the mutation network $\network_\gamma$ are given by
	\begin{align}
	N_0 = 
	\begin{pmatrix}
	2 & 0 \\
	0 & 2 
	\end{pmatrix},
	N_+ = 
	\begin{pmatrix}
	0 & 0 \\
	0 & 0
	\end{pmatrix},
	N_- = 
	\begin{pmatrix}
	0 & 1 \\
	1 & 0 
	\end{pmatrix},
	\end{align}
	and the Neumann-Zagier matrices of the mutation loop $\gamma$ are given by
	\begin{align}
	A_+ = 
	\begin{pmatrix}
	2 & 0 \\
	0 & 2
	\end{pmatrix},
	A_- = 
	\begin{pmatrix}
	2 & -1  \\
	-1 & 2 
	\end{pmatrix}.
	\end{align}
\end{example}

Next, we give some lemmas relating to the adjacency matrices of mutation networks.
Here, $\delta$ is the Kronecker delta.
\begin{lemma}\label{lemma: n0}
	For each $e \in \mathcal{E}$ and $t=1,\dots, T$,
	\begin{align}\label{eq: n^0}
		n_{et}^0 =  \sum_{(i,u) \in e} \delta_{i,m_t} \delta_{u,t-1} + \delta_{i,\sigma_t(m_t)} \delta_{tu}.
	\end{align}
	In particular, if $e$ is the equivalence class of $(m_s ,s-1)$ for some $s=1,\dots, T$,
	we have
	\begin{align}
		n_{et}^0 =  \delta_{st} + \sum_{ (i,u) \in e} \delta_{i,\sigma_t(m_t)} \delta_{tu}.
	\end{align}
\end{lemma}
\begin{proof}
	It follows from the rule for drawing broken lines that $n _{et}^0 = x_{et} + y_{et}$ where
	\begin{align*}
		&x_{et} = 
		\begin{cases}
			1 & \text{if $(m_t, t-1) \in e$}, \\
			0 & \text{otherwise}, \\
		\end{cases} \\
		&y_{et} = 
		\begin{cases}
			1 & \text{if $(\sigma_t(m_t), t) \in e$}, \\
			0 & \text{otherwise}. \\
		\end{cases}
	\end{align*}
	The first part of the lemma now follows from the fact that
	\begin{align*}
		&x_{et} =\sum_{(i,u) \in e} \delta_{i,m_t} \delta_{u,t-1}, \\
		&y_{et}=  \sum_{ (i,u) \in e} \delta_{i,\sigma_t(m_t)} \delta_{tu}.
	\end{align*}
	Since $(m_t , t-1) \in  [(m_s , s-1)]$ if and only if $s=t$, we obtain $x_{et} = \delta_{st}$ when $e =  [(m_s , s-1)]$.
\end{proof}

\begin{lemma}\label{lemma: n epsilon}
	For all $\varepsilon \in \{ \pm \}$, $e \in \mathcal{E}$, and $t=1 ,\dots, T$,
	\begin{align}
	\label{eq: n^epsilon 1}
	n_{et}^{\varepsilon} 
	&= \sum_{ (i,u) \in e}
	\maxzero{ \varepsilon B_{m_t ,i}(t-1)} \delta_{t,u+1} \\
	\label{eq: n^epsilon 2}
	&= \sum_{ (i,u) \in e}
	\maxzero{ \varepsilon B_{m_t ,\sigma_t^{-1}(i)}(t-1)} \delta_{tu}.
	\end{align}
\end{lemma}
\begin{proof}
	For simplicity, we assume that $\varepsilon = +$;
	the proof for $\varepsilon = -$ is similar.
	In the graph $\overline{\network}_\gamma$,
	there are
	\begin{align*}
		\maxzero{ B_{m_t , i} (t-1) }
	\end{align*}
	arrows from the square vertex labeled $t$ to the black vertex labeled $(i,t-1)$, so we have
	\begin{align*}
		n_{et}^+ = \sum_{ (i',u') \in e}
			\maxzero{  B_{m_t ,i'}(t-1)} \delta_{t,u'+1}.
	\end{align*}
	For any element $(i', u') \in \overline{\mathcal{E}}$ such that $u' \neq T$, we can set $(i,u) = (\sigma_{u' + 1}(i') , u' +1)$, so $(i',u') \sim (i,u)$ unless $i' = m_u$.
	Thus, we can compute $n_{et}^+$ as follows:
	\begin{align*}
		n_{et}^+ &= \sum_{ (i',u') \in e}
		\maxzero{  B_{m_t ,i'}(t-1)} \delta_{t,u'+1} \\
		&=\sum_{ \substack{(i',u') \in e \\ i' \neq m_t }}
		\maxzero{  B_{m_t ,i'}(t-1)} \delta_{t,u'+1} \\
		&=\sum_{ \substack{(i,u) \in e \\ i \neq \sigma_t^{-1} (m_t) }}
		\maxzero{  B_{m_t ,\sigma_t^{-1}(i)} (t-1)} \delta_{tu} \\
		&=\sum_{(i,u) \in e  }
		\maxzero{  B_{m_t ,\sigma_t^{-1}(i)} (t-1)} \delta_{tu} .
	\end{align*}
\end{proof}

\subsection{Fully mutated mutation sequences}\label{subsection: fully mutated}
We say that the mutation sequence $\gamma=(B,m,\sigma)$ is \emph{fully mutated} if, for any $e \in \mathcal{E}$, there exists a $t \in \{ 1, \dots , T \}$ such that $(m_t , t-1) \in e$.
When $\sigma_t = \id$ for all $t=1 ,\dots, T$, the mutation sequence $\gamma$ is fully mutated if and only if $\{ m_1 ,\dots , m_T \} = \{ 1, \dots, n \}$.

For a given permutation $\sigma$ of the set $\{ 1, \dots, n \}$, we define its permutation matrix by $P_\sigma = (\delta_{i\sigma(j)})_{1 \leq i,j \leq n}$.
For all $k = 1, \dots , n$, let $H_k$ be the $n \times n$ diagonal matrix that is identical to the identity matrix except that the $k$-th entry is $0$.
Define the block matrix $\alpha_{m,\sigma}$ by
\begin{align}\label{eq: def of alpha}
\alpha_{m,\sigma} =
\begin{pmatrix}
O &  \cdots & O &  P_{\sigma_1} H_{m_1} \\
P_{\sigma_2} H_{m_2} & \cdots & O & O \\
\vdots &  \ddots & \vdots & \vdots \\
O  &  \cdots &  P_{\sigma_T} H_{m_T} &O \\
\end{pmatrix}.
\end{align}
The matrix $\alpha_{m,\sigma}$ is a $T \times T$ block matrix whose blocks are $n \times n$ matrices.
In the following, we assume any references to its $0$-th block mean its $T$-th block.
Let $(i,t)$ denote the $i$-th index in the $t$-th block.

Let $G_{\alpha_{m,\sigma}}$ be the graph whose adjacency matrix is $\alpha_{m,\sigma}$, that is, there is an arrow from $(i,s)$ to $(j,t)$ if $(\alpha_{m,\sigma})_{(i,s),(j,t)} = 1$, but not if $(\alpha_{m,\sigma})_{(i,s),(j,t)} = 0$.
Such graphs consist of directed paths and directed cycles, and an example is shown in Figure \ref{fig: G alpha}.
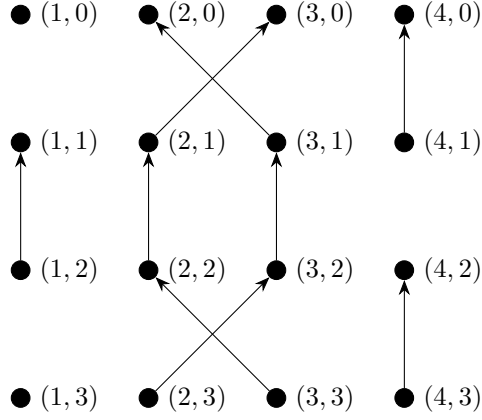
\begin{figure}
	\centering
	\begin{tikzpicture}
	[scale=1.7,auto=left,black_vertex/.style={circle,draw,fill,scale=0.75},square_vertex/.style={rectangle,scale=0.8,draw}]
	\node (10) at (1,0) [black_vertex][label=right:${(1,0)}$]{};
	\node (20) at (2,0) [black_vertex][label=right:${(2,0)}$]{};
	\node (30) at (3,0) [black_vertex][label=right:${(3,0)}$]{};
	\node (40) at (4,0) [black_vertex][label=right:${(4,0)}$]{};
	\node (11) at (1,-1) [black_vertex][label=right:${(1,1)}$]{};
	\node (21) at (2,-1) [black_vertex][label=right:${(2,1)}$]{};
	\node (31) at (3,-1) [black_vertex][label=right:${(3,1)}$]{};
	\node (41) at (4,-1) [black_vertex][label=right:${(4,1)}$]{};
	\node (12) at (1,-2) [black_vertex][label=right:${(1,2)}$]{};
	\node (22) at (2,-2) [black_vertex][label=right:${(2,2)}$]{};
	\node (32) at (3,-2) [black_vertex][label=right:${(3,2)}$]{};
	\node (42) at (4,-2) [black_vertex][label=right:${(4,2)}$]{};
	\node (13) at (1,-3) [black_vertex][label=right:${(1,3)}$]{};
	\node (23) at (2,-3) [black_vertex][label=right:${(2,3)}$]{};
	\node (33) at (3,-3) [black_vertex][label=right:${(3,3)}$]{};
	\node (43) at (4,-3) [black_vertex][label=right:${(4,3)}$]{};
	\draw[arrows={-Stealth[scale=1.2]}] (41)--(40);
	\draw[arrows={-Stealth[scale=1.2]}] (31)--(20);
	\draw[arrows={-Stealth[scale=1.2]}] (21)--(30);
	\draw[arrows={-Stealth[scale=1.2]}] (12)--(11);
	\draw[arrows={-Stealth[scale=1.2]}] (22)--(21);
	\draw[arrows={-Stealth[scale=1.2]}] (32)--(31);
	\draw[arrows={-Stealth[scale=1.2]}] (33)--(22);
	\draw[arrows={-Stealth[scale=1.2]}] (23)--(32);
	\draw[arrows={-Stealth[scale=1.2]}] (43)--(42);
	\end{tikzpicture}
	\caption{Example graph $G_{\alpha_{m,\sigma}}$, where $m = (1,4,1)$ and $ \sigma=( (2 \; 3) , \id , (2 \; 3)) $. Here, we identify $(i,3)$ with $(i,0)$ for all $i=1,2,3,4$. There are four connected components, two directed paths and two directed cycles.}
	\label{fig: G alpha}
\end{figure}

\begin{lemma}\label{lemma: nilpotent}
	For a mutation sequence $\gamma = (B,m,\sigma)$ of length $T$,
	The following four statements are equivalent.
	\begin{enumerate}
		\item $\gamma$ is fully mutated.
		\item $\network_\gamma$ has $T$ black vertices.
		\item $\alpha_{m,\sigma}$ is nilpotent; that is, there exists a positive integer $k$ such that $(\alpha_{m,\sigma})^k = 0$.
		\item $G_{\alpha_{m,\sigma}}$ has no directed cycles.
	\end{enumerate}
\end{lemma}
\begin{proof}
	It follows from the definition of the equivalence relation on $\overline{\mathcal{E}}$ that $(i,s)$ and $(j,t)$ lie in the same component of $G_{\alpha_{m,\sigma}}$ if and only if $(i,s)$ and $(j,t)$ belong to the same equivalence class in $\mathcal{E}$.
	Thus, we can identify the set of connected components of $G_{\alpha_{m,\sigma}}$ with $\mathcal{E}$.
	
	The map
	$\{ 1, \dots , T \} \to \mathcal{E}$
	that sends $t$ to $[( m_t , t-1 )]$ is injective, because $( m_t , t-1 )$ is the starting point of some directed path in $G_{\alpha_{m,\sigma}}$.
	It is also surjective if and only if $\gamma$ is fully mutated, showing that $\text{(i)}  \Leftrightarrow \text{(ii)}$.
	All directed paths start with $(m_t , t-1)$ for some $1\leq t \leq T$, implying that there are a total of $T$ directed paths,
	and hence that $\text{(ii)}  \Leftrightarrow \text{(iv)}$.
	Finally, since $G_{\alpha_{m,\sigma}}$ consists of directed paths and directed cycles, we have $\text{(iii)}  \Leftrightarrow \text{(iv)}$.
\end{proof}

If $\gamma$ is fully-mutated, the map
$\{ 1, \dots , T \} \to \mathcal{E}$
that sends $t$ to $[( m_t , t-1 )]$ is a bijection.
Using this bijection, we assume that $\mathcal{E} \times T$ matrices such as $A_+$ and $A_-$ are $T \times T$ matrices when $\gamma$ is fully mutated.

\subsection{Gluing equations of mutation networks}\label{subsection: gluing equation}
In this section we see that the Neumann-Zagier matrices describe the fixed point equation of the cluster transformation.
Let $\gamma$ be a mutation sequence.
For $t=1, \dots , T$, define a rational functions $z_{t,+}, z_{t,-} \in \univsf$ by
\begin{align}\label{eq: def of z}
z_{t,+} =  \frac{Y_{m_t}}{Y_{m_t}+1}  , \; 
z_{t,-}  =  \frac{1}{Y_{m_t}+1}  ,
\end{align}
where we denote $Y_{m_t}(t-1)$ by $Y_{m_t}$.

Define the set $\gespace$ by
\begin{align*}
\gespace = \{ (z_{t,+} , z_{t,-})_{1 \leq t \leq T} \in (\CC \setminus \{ 0,1 \})^{2T}  \mid z_+ + z_- = 1 , z_+^{-A_+} z_-^{A_-} =1  \}
\end{align*}
where $z_+ + z_- = 1$ means $z_{t,+} + z_{t,-} = 1$ for all $t=1, \dots , T$, and $z_+^{-A_+} z_-^{A_-} =1 $ is the shorthand for the following equations:
\begin{align}\label{eq: gluing equation}
\prod_{t=1}^T {z_{t,+}}^{-a_{et}^+ } {z_{t,-}}^{a_{et}^-} =1 \quad\quad \text{for each $e \in \mathcal{E}$.}
\end{align}
We call \eqref{eq: gluing equation} the \emph{gluing equations of the mutation network $\network_\gamma$}.
The solution space $\gespace$ of these gluing equations is an affine algebraic set because
\begin{align*}
&\gespace = \bigl\{  (z_{t,+} ,z_{t,+}^{-1} ,(1-z_{t,+})^{-1},z_{t,-} ,z_{t,-}^{-1} ,(1-z_{t,-})^{-1} )_{1 \leq t \leq T} \in \CC^{6T} \\
&\quad\quad\quad \mid \text{$f_t = g_e =h_{t,\varepsilon} =i_{t,\varepsilon}=0$ for all $t = 1,\dots, T$, $e \in \mathcal{E}$, and $\varepsilon \in \{ \pm \}$} \bigr\},
\end{align*}
where
\begin{align*}
&f_t = z_{t,+} + z_{t,-} -1 , \; g_e = \prod_{t=1}^T {z_{t,+}}^{-a_{et}^+ } {z_{t,-}}^{a_{et}^-} -1 , \\
&h_{t,\varepsilon} = z_{t,\varepsilon} z_{t,\varepsilon}^{-1} -1 ,\; i_{t,\varepsilon} =(1- z_{t,\varepsilon})(1- z_{t,\varepsilon})^{-1} -1 .
\end{align*}

The gluing equations of the mutation network $\network_\gamma$ are related to the fixed-point equation $\mu_\gamma (y) = y$ of the cluster transformation of the mutation sequence $\gamma$.
Let $\fpspace$ be the space of the fixed-points of $\mu_\gamma$ for which $Y_{m_t}(t-1)$ is neither $0$ nor $-1$ for any $t$:
\begin{align*}
\fpspace = \{ \eta= (\eta_1 , \dots , \eta_n) \in \CC  \mid  \text{$\mu_\gamma (\eta) = \eta$ ,$Y_{m_t}|_{y=\eta} \neq 0,-1$ for $t=1,\dots ,T$} \},
\end{align*}
where we denote $Y_{m_t}(t-1)$ by $Y_{m_t}$.
We can define an affine algebraic set structure on $\fpspace$ as in the case of $\gespace$.
\begin{proposition}\label{prop: bijection F D}
	Let $\gamma$ be a fully mutated mutation sequence.
	The map
	\begin{align*}
	\varphi : \fpspace \to \gespace
	\end{align*}
	that sends $\eta \in \fpspace$ to 
	\begin{align*}
	\left. \left( \frac{Y_{m_t}  }{Y_{m_t} +1 }, \frac{1}{Y_{m_t} +1 } \right)_{1 \leq  t \leq T} \right|_{y=\eta}
	\end{align*}
	is an isomorphism between affine algebraic sets.
\end{proposition}
\begin{proof}
	First, we prove that $\varphi(\eta) \in \gespace$.
	Let $e \in \mathcal{E}$.
	Then there exist $s, s' \in \{ 1, \dots , T \}$ such that $(m_s , s-1 ) , (\sigma_{s'} (m_{s'}) ,s') \in e$, and these $s,s'$ are unique.
	From the mutation rule \eqref{eq: Y mutation}, we have
	\begin{align*}
	Y_{m_s} (s) = Y_{m_{s'}}(s')^{-1} \prod_{\substack{(i,u) \in e \\ u\neq T}} {z_{u+1,+}}^{\maxzero{B_{m_{u+1},i} (u) }} {z_{u+1,-}}^{-\maxzero{-B_{m_{u+1},i} (u) }}.
	\end{align*}
	Since
	\begin{align*}
	\prod_{\substack{(i,u) \in e \\ u\neq T}} {z_{u+1,+}}^{\maxzero{B_{m_{u+1},i} (u) }} {z_{u+1,-}}^{-\maxzero{-B_{m_{u+1},i} (u) }} = \prod_{t=1}^T {z_{t,+}}^{n_{et}^+} {z_{t,-}}^{-n_{et}^-}
	\end{align*}
	by \eqref{eq: n^epsilon 1} and
	\begin{align*}
	Y_{m_s}(s)^{-1} Y_{m_{s'}}(s')^{-1} =  \prod_{t=1}^T {z_{t,+}}^{-n_{et}^0} {z_{t,-}}^{n_{et}^0}
	\end{align*}
	by \eqref{eq: n^0},
	we obtain
	\begin{align*}
	\prod_{t=1}^T {z_{t,+}}^{-a_{et}^+ } {z_{t,-}}^{a_{et}^-} =1.
	\end{align*}
	Thus $\varphi(\eta) \in \gespace$.
	
	Next, we construct the inverse of $\varphi$.
	Since $\gamma$ is fully mutated, for any $1 \leq j \leq n$ and $1 \leq t \leq T$ there must exist a unique $s \in \{1, \dots ,T \}$ such that $(j,t) \sim (\sigma_s(m_s) , s)$.
	From the mutation rule \eqref{eq: Y mutation}, we have
	\begin{align*}
	Y_j(t) = z_{s,+}^{-1}z_{s,-}
	\prod_{i=1}^l {z_{s+i,+}}^{\maxzero{b_i}} {z_{s+i,-}}^{-\maxzero{-b_i}} 
	\end{align*}
	where $l$ is the distance between $(j,t)$ and $(\sigma_s(m_s) , s)$ in the graph $G_{\alpha_{m,\sigma}}$ (see Section \ref{subsection: fully mutated}) and
	\begin{align*}
	b_i = B_{m_{s+i},\sigma_{s+i-1} \cdots \sigma_s(m_s) } (s+i-1).
	\end{align*}
	Here, we have assumed that $0 \leq s+i-1 \leq T-1$ by working modulo $T$.
	This gives the inverse of $\varphi$, completing the proof.
\end{proof}

\begin{corollary}
	The isomorphism $\varphi$ restricts to the bijection
	\begin{align*}
		\varphi : \fpspace^{>0} \to \gespace^{(0,1)},
	\end{align*}
	where $\fpspace^{>0} = \fpspace \cap \R_{>0}^{n}$ and $\gespace^{(0,1)} = \gespace \cap (0,1)^{2T}$.
\end{corollary}

The next two sections give examples of mutation sequences in which the solution spaces $\gespace$ play an important role.

\subsection{Surface examples}\label{section: surface examples}
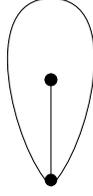
\begin{figure}
	\begin{tikzpicture}[scale=0.5]
	\node (v1) at (-1,-3) {};
	\node at (-1,-6) {};densely dashed
	\draw (-1,-6) .. controls (0,-5) and (1,-1) .. (-1,-1) .. controls (-3,-1) and (-2,-5) .. (-1,-6) node (v2) {};
	\draw[*-*]  (v1.center) edge (v2.center);
	\end{tikzpicture}
	\caption{A self-folded edge.}
	\label{figure: self-folded}
\end{figure}
In this section, we review mutation loops related to hyperbolic 3-manifolds according to~\cite{NagaoTerashimaYamazaki}.
Let $S$ be a closed oriented surface and $P$ be a finite set of points on $S$.
The points in $P$ is called \emph{punctures}.
Let $\Gamma$ be a ideal triangulation of $(S,P)$, that is, a triangulation of $S$ whose vertices are points in $M$.
We assume that $\Gamma$ has no self-folded edges (see Figure \ref{figure: self-folded}).
We define a skew-symmetric matrix $B_{\Gamma}$ as follows.
For a triangle $t$ in $\Gamma$ and edges $i$ and $j$,
we define a skew-symmetric matrix $B^{t}$ by
\begin{align*}
B_{ij}^t = 
\begin{cases}
1 & \text{if $t$ has sides $i$ and $j$, and the direction from $i$ to $j$ is clockwise},\\
-1 &\text{if the same holds, but is counter-clockwise},\\
0 &\text{otherwise}.
\end{cases}
\end{align*}
We define $B_{\Gamma}$ by
\begin{align*}
B_{\Gamma} = \sum_{t \in \Gamma} B^t,
\end{align*}
where the sum is taken over all triangles in $\Gamma$.

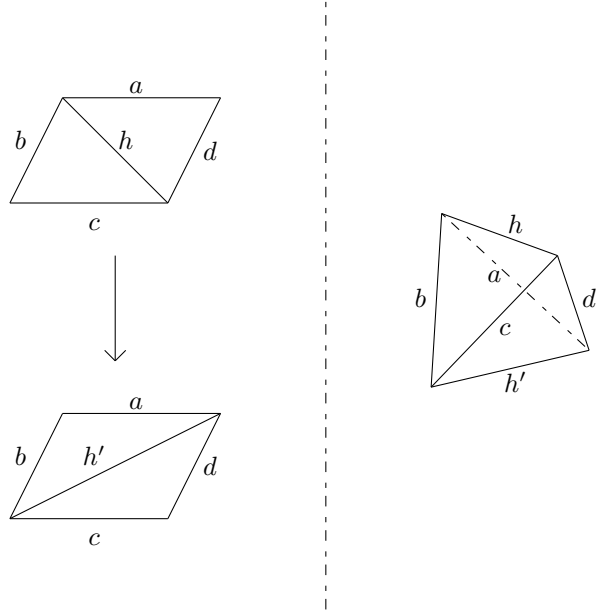
\begin{figure}
	\begin{tikzpicture}[scale=0.7]
	\node (v2) at (-3,4) {};
	\node (v3) at (0,4) {};
	\node (v1) at (-4,2) {};
	\node (v4) at (-1,2) {};
	\node (v6) at (-4,-4) {};
	\node (v5) at (-3,-2) {};
	\node (v8) at (0,-2) {};
	\node (v7) at (-1,-4) {};
	\node (v9) at (-2,1) {};
	\node (v10) at (-2,-1) {};
	\draw  (v1.center) edge (v2.center);
	\draw  (v2.center) edge (v3.center);
	\draw  (v3.center) edge (v4.center);
	\draw  (v1.center) edge (v4.center);
	\draw  (v2.center) edge (v4.center);
	\draw  (v5.center) edge (v6.center);
	\draw  (v6.center) edge (v7.center);
	\draw  (v7.center) edge (v8.center);
	\draw  (v8.center) edge (v5.center);
	\draw  (v6.center) edge (v8.center);
	\draw  (v9.center) edge (v10.center);
	\node (v11) at (2,6) {};
	\node (v12) at (2,-6) {};
	\draw [dash pattern=on 2pt off 3pt on 4pt off 4pt] (v11) edge (v12);
	\node (v13) at (4.2,1.8) {};
	\node (v14) at (4,-1.5) {};
	\node (v16) at (6.4,1) {};
	\node (v15) at (7,-0.8) {};
	\draw (v13.center) edge (v14.center);
	\draw (v14.center) edge (v15.center);
	\draw (v16.center) edge (v15.center);
	\draw (v13.center) edge (v16.center);
	\draw [dash pattern=on 2pt off 3pt on 4pt off 4pt] (v13.center) edge (v15.center);
	\draw (v14.center) edge (v16.center);
	\node (v17) at (-2.2,-0.8) {};
	\node (v18) at (-1.8,-0.8) {};
	\draw (v17.center) edge (v10.center);
	\draw (v10.center) edge (v18.center);
	\node at (-1.6,4.2) {$a$};
	\node at (-3.8,3.2) {$b$};
	\node at (-2.4,1.6) {$c$};
	\node at (-1.8,3.2) {$h$};
	\node at (-0.2,3) {$d$};
	\node at (-1.6,-1.8) {$a$};
	\node at (-3.8,-2.8) {$b$};
	\node at (-2.4,-2.8) {$h'$};
	\node at (-2.4,-4.4) {$c$};
	\node at (-0.2,-3) {$d$};
	\node at (5.2,0.6) {$a$};
	\node at (3.8,0.2) {$b$};
	\node at (5.4,-0.4) {$c$};
	\node at (7,0.2) {$d$};
	\node at (5.6,1.6) {$h$};
	\node at (5.6,-1.4) {$h'$};
	\end{tikzpicture}
	\caption{Attaching a topological tetrahedron to quadrilaterals.}
	\label{figure: tetrahedron}
\end{figure}

Let $\phi : S \to S$ be a homeomorphism that restricts a bijection on $P$.
Then $\phi(\Gamma)$ is an ideal triangulation of $S$.
Let us choose a sequence of flips and a permutation of edge labels that connects $\Gamma$ and $\phi(\Gamma)$.
This provides a mutation loop $\gamma$ such that $B(0)=B_\Gamma$ and $B(T)=B_{\phi(\Gamma)}$ by the compatibility between flips and mutations.
We assume that triangulations that appear in the sequence of flips do not contain self-folded edges, and also assume that $\gamma$ is fully mutated.
Let $h$ be a edge flipped in the mutation loop,
and $h'$ be the edge after the flip.
Let $a$, $b$, $c$, and $d$ be the edges of the quadrilateral in the triangulations whose diagonals are $h$ and $h'$.
We attach a topological tetrahedron to this quadrilateral
as in Figure \ref{figure: tetrahedron}.
This provides a topological tetrahedral decomposition $\Delta$ of the mapping torus of $\phi$.
More specifically, removing 0-simplices from $\Delta$ gives an topological ideal triangulation of the 3-manifold $(S\setminus P) \times [0,1] / \{ (x,0) \sim (\phi(x),1) \}$.
Let $A_+$ and $A_-$ be the Neumann-Zagier matrices of this mutation loop.
Then these matrices coincides with the matrices used by Neumann-Zagier~\cite[(16)]{NeumannZagier} to describe the edge gluing equations for a 3-manifold (the mapping torus of $\phi$ in this case) to be hyperbolic.

\begin{example}
	Let us consider the examples in~\cite[Section 5.1]{NagaoTerashimaYamazaki}.
	Let $S$ be a torus, and $P$ is a one point set.
	Let $\Gamma$ be a triangulation of $(S,P)$.
	Then we have
	\begin{align*}
	B_{\Gamma} =
	\begin{pmatrix}
	0 & 2 & -2 \\
	-2 & 0 & 2 \\
	2 & -2 & 0
	\end{pmatrix}
	\end{align*}
	for an appropriate labeling of the edges of $\Gamma$.
	The mapping class group of $S \setminus P \simeq (\mathbb{R}^2 \setminus \mathbb{Z}^2) / \mathbb{Z}^2$ is identified with $\mathrm{SL}_2(\mathbb{Z})$.
	We define the following two elements in $\mathrm{SL}_2(\mathbb{Z})$:
	\begin{align*}
		R=
		\begin{pmatrix}
			1 & 1 \\
			0 & 1
		\end{pmatrix},\quad
		L=
		\begin{pmatrix}
			1 & 0 \\
			1 & 1
		\end{pmatrix}.
	\end{align*}
	Let $\phi$ be a homeomorphism that represents $RL \in \mathrm{SL}_2(\mathbb{Z})$.
	Then the mapping torus of $\phi$ is the figure-eight knot complement in the 3-sphere (see Figure \ref{figure: 4_1}).
	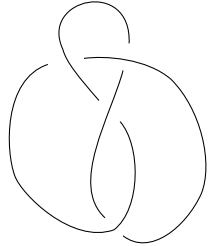
\begin{figure}
		\begin{tikzpicture}[scale=0.4]
		\draw (2.5,2.3) .. controls (2.6,4.5) and (-0.5,3.8) .. (0.3,2.1) .. controls (0.4,1.7) and (0.8,1.2) .. (1.5,0.4);
		\draw (2.2,-0.3) .. controls (2.9,-1.1) and (2.9,-3.2) .. (2,-3.9) .. controls (0.6,-4.4) and (-1.1,-2.7) .. (-1.3,-2.1) .. controls (-1.6,-1.3) and (-1.7,1.1) .. (-0.2,1.6);
		\draw (1.7,-3.5) .. controls (0.5,-2.4) and (1.9,-0.1) .. (2.3,1.4);
		\draw (1,1.8) .. controls (1.7,1.9) and (3.1,1.8) .. (3.9,1.1) .. controls (4.9,0.1) and (5.3,-1.7) .. (4.9,-2.7) .. controls (4.4,-3.8) and (3.2,-4.8) .. (2.3,-4.1);
		\end{tikzpicture}
		\caption{The figure-eight knot.}
		\label{figure: 4_1}
	\end{figure}
	
	We can see that the mutation loop
	$\gamma = (B, m,\sigma)$ for
	$m = (2,1),\sigma =(\id, (3 \; 2\; 1))$ has the property that we want, that is,
	$B(0)=B_\Gamma$ and $B(2)=B_{\phi(\Gamma)}$.
	Moreover, triangulations that appear in this mutation sequence do not have self-folded edges, and $\gamma$ is fully mutated.
	See~\cite[Figure 16]{NagaoTerashimaYamazaki}.
	The mutation network $\network_\gamma$ is given by
	\begin{align*}
	\network_\gamma = 
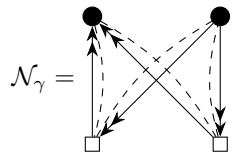
	\begin{tikzpicture}
	[baseline={([yshift=-.5ex]current bounding box.center)},scale=1.7,auto=left,black_vertex/.style={circle,draw,fill,scale=0.75},square_vertex/.style={rectangle,scale=0.8,draw}]
	\node (a) at (0,0) [black_vertex]{};
	\node (b) at (1,0) [black_vertex]{}; 
	\node (t1) at (0,-1) [square_vertex]{};
	\node (t2) at (1,-1) [square_vertex]{};
	\draw [dashed] (a)edge [bend left=15](t1);
	\draw [dashed] (t1)edge [bend left=15](b);
	\draw [dashed] (a)edge [bend left=15](t2);
	\draw [dashed] (t2)edge [bend left=15](b);
	\draw[arrows={-Stealth[scale=1.2]Stealth[scale=1.2]}] (b)--(t1);
	\draw[arrows={-Stealth[scale=1.2]Stealth[scale=1.2]}] (t1)--(a);
	\draw[arrows={-Stealth[scale=1.2]Stealth[scale=1.2]}] (b)--(t2);
	\draw[arrows={-Stealth[scale=1.2]Stealth[scale=1.2]}] (t2)--(a);
	\end{tikzpicture}
	\end{align*}
	(this mutation network is the same as that in~\cite[Figure 21]{TerashimaYamazaki}), and thus we obtain
	\begin{align*}
	N_0=
	\begin{pmatrix}
	1 & 1 \\
	1 & 1
	\end{pmatrix},
	N_+=
	\begin{pmatrix}
	2 & 2 \\
	0 & 0
	\end{pmatrix},
	N_-=
	\begin{pmatrix}
	0 & 0 \\
	2 & 2
	\end{pmatrix}.
	\end{align*}
	Therefore, the Neumann-Zagier matrices of this mutation loop are
	\begin{align*}
	A_+=
	\begin{pmatrix}
	-1 & -1\\
	1 & 1
	\end{pmatrix},
	A_-=
	\begin{pmatrix}
	1 & 1 \\
	-1 & -1
	\end{pmatrix}.
	\end{align*}
	
	The solution space of the gluing equations is given by
	\begin{align*}
	\mathfrak{D}_\gamma =
	\{ (z,w) \in (\CC \setminus \{ 0,1 \})^2 \mid
	z w (1-z) (1-w) = 1
	\},
	\end{align*} 
	where $z = z_{1,-}$ and $w=z_{2,-}$.
	An element $(z,w) \in \mathfrak{D}_\gamma$ that satisfies $\Imag(z),\Imag(w)>0$ and $\arg(z)+\arg(w) + \arg(1/(1-z))+\arg(1/(1-w))= 2\pi$
	gives a hyperbolic structure of the figure-eight knot complement,
	and the distinguished solution 
	\begin{align*}
	z=w=\frac{1+\sqrt{-3}}{2}
	\end{align*}
	gives the complete hyperbolic structure~\cite[Section 4.3]{Thurston}.
	In particular, the hyperbolic volume of the figure-eight knot complement is
	\begin{align*}
	2 D \left(\frac{1+\sqrt{-3}}{2} \right)  = 2.02988321281 \dots
	\end{align*}
	where $D(z)$ is the Bloch-Wigner function, which is defined by $D(z) = \Imag(\Li_2(z)) + \arg(1-z) \log \lvert z \rvert$ ($\Li_2(z)$ is the dilogarithm function).
	The fixed-point $\eta \in \mathfrak{F}_\gamma$ corresponding to this distinguished solution via the bijection in Proposition \ref{prop: bijection F D} is the one computed in~\cite[Section 5.1]{NagaoTerashimaYamazaki}, and this is given by
	\begin{align*}
	\eta_1=1, \quad \eta_2 = \frac{-1-\sqrt{-3}}{2}, \quad \eta_3 = \frac{-1+\sqrt{-3}}{2}.
	\end{align*}
\end{example}

\subsection{Dynkin examples}\label{section: dynkin examples}
Let $X_n$ be an $ADE$ Dynkin diagram with $n$ nodes,
and $Q$ be a quiver such that its underlying graph is $X_n$ and all vertices in $Q$ are sinks or sources.
Let $B$ be the skew-symmetric matrix corresponding to $Q$.
Let $m_{\pm}$ be a sequence of all the sources and sinks respectively.
Then $\gamma = (B, m_- m_+, (\id, \dots, \id))$ is a mutation loop.
For example, $\gamma$ for $A_2$ is the mutation loop considered in Example \ref{example: A2} and \ref{example: A2 matrix}.
Note that the length of $\gamma$ is equal to the number of nodes in the Dynkin diagram, that is, we have $T=n$.

The Neumann-Zagier matrices $A_+$ and $A_-$ are given by
\begin{align*}
	A_+ = 2 I_n, \quad A_- = C,
\end{align*}
where $I_n$ is the identity matrix and $C$ is a Cartan matrix of $X_n$.
The solution space of the gluing equations in the open interval $(0,1)$ is given by
\begin{align*}
\mathfrak{D}_\gamma^{(0,1)} &=
\{ (z_{t,+},z_{t,-})_{1 \leq t \leq n} \in (0,1)^{2n} \mid z_{s,+}^{-2} \prod_{t=1}^n z_{t,-}^{C_{st}}  = 1 \text{ for $s=1,\dots, n$}
\}\\
&=\{ z \in (0,1)^n  \mid  \prod_{t=1}^n z_t^{C_{st}/2}  = 1-z_s \text{ for $s=1,\dots, n$}
\}
\end{align*}
where $z_t = z_{t,-}$.
This space has exactly one element (this follows from \cite[Lemma 2.1]{MashaSander} and the fact that the Cartan matrix $C$ is positive definite),
and we denote this element by $\zeta$.
Then we have the following dilogarithm identity:
\begin{align}\label{eq: dilog central charge}
	\frac{6}{\pi^2} \sum_{i=1}^n L(\zeta_i) =  \frac{2 \dim \mathfrak{g}}{2 + \dcn} - n,
\end{align}
where $L(x)$ is the Rogers dilogarithm defined by $L(x)=\Li_2(x) + \frac{1}{2} \log(x) \log(1-x)$,
$\mathfrak{g}$ is the simple Lie algebra of the type $X_n$, and $\dcn$ is the dual Coxeter number of $X_n$.
The right-hand side of \eqref{eq: dilog central charge} is a central charge of a conformal field theory associated with $\mathfrak{g}$.
See~\cite[Section 5.1]{KNS2011} for more detail on the dilogarithm identities.

\subsection{Symplectic property of mutation loops}
As we saw in the last two sections, we often encounter mutation loops, not just mutation sequences.
Here, we give a necessary condition for a mutation sequence to be a mutation loop when the exchange matrices are skew-symmetric.
\begin{proposition}
	Let $\gamma = (B,m,\sigma)$ be a mutation sequence, and
	suppose that $B$ is skew-symmetric.
	If $B(0)=B(T)$, then $A_+ A_-^T$ is symmetric.
\end{proposition}
\begin{proof}
	What we have to show is
	\begin{align*}
		\sum_{t=1}^T (n_{et}^+ n_{ft}^0 + n_{et}^0 n_{ft}^- + n_{et}^- n_{ft}^+ )
		= \sum_{t=1}^T (n_{et}^- n_{ft}^0 + n_{et}^0 n_{ft}^+ + n_{et}^+ n_{ft}^- )
	\end{align*}
	for each $e  ,f \in \mathcal{E}$.
	First, define the number $X(t)$ by
	\begin{align*}
		X(t) = \sum_{ \substack{(i,s ) \in e \\ (j,u) \in f}} B_{ij}(t) \delta_{st} \delta_{ut}.
	\end{align*}
	Next, if we let $k=m_s$,
	then we find
	\begin{align*}
		&X(t) - X(t-1) = \sum_{\substack{(i,s ) \in e \\ (j,u) \in f \\ i,j  \neq k}} (B_{\sigma_t(i) \sigma_t(j)} (t) - B_{ij}(t-1) )\delta_{s+1,t} \delta_{u+1,t} \\
		& + \sum_{\substack{(i,s ) \in e \\ (j,u) \in f \\ j  \neq k}} B_{\sigma_t( k ) \sigma_t (j)} (t) \delta_{i,\sigma_t(k)} \delta_{st} \delta_{u+1,t} - \sum_{\substack{(i,s ) \in e \\ (j,u) \in f \\ j  \neq k}} B_{kj}  (t-1) \delta_{ik} \delta_{s+1,t} \delta_{u+1,t} \\
		& + \sum_{\substack{(i,s ) \in e \\ (j,u) \in f \\ i  \neq k}} B_{\sigma_t (i)\sigma_t( k ) } (t)\delta_{j,\sigma_t (k) }\delta_{s+1,t} \delta_{ut} - \sum_{\substack{(i,s ) \in e \\ (j,u) \in f \\ i \neq k}} B_{ik}  (t-1)  \delta_{ji}\delta_{s+1,t} \delta_{u+1,t}.
	\end{align*}
	
	Using Lemma  \ref{lemma: n epsilon}, we obtain
	\begin{align*}
		&\sum_{\substack{(i,s ) \in e \\ (j,u) \in f \\ i,j  \neq k}} (B_{\sigma_t(i) \sigma_t(j)} (t) - B_{ij}(t-1)) \delta_{s+1,t} \delta_{u+1,t} \\
		=& \sum_{\substack{(i,s ) \in e \\ (j,u) \in f }}  \bigl( \maxzero{B_{ik}(t-1)} \maxzero{B_{kj}(t-1)}- \maxzero{- B_{ik}(t-1)} \maxzero{-B_{kj}(t-1)} \bigr) \delta_{s+1,t} \delta_{u+1,t} \\
		=&\sum_{\substack{(i,s ) \in e \\ (j,u) \in f }}  \bigl( \maxzero{-B_{ki}(t-1)} \maxzero{B_{kj}(t-1)}- \maxzero{B_{ki}(t-1)} \maxzero{-B_{kj}(t-1)} \bigr) \delta_{s+1,t} \delta_{u+1,t} \\
		=&n_{et}^- n_{ft}^+  - n_{et}^+ n_{ft}^-.
	\end{align*}
	Using Lemma \ref{lemma: n0}, we find
	\begin{align*}
	&\sum_{\substack{(i,s ) \in e \\ (j,u) \in f \\ j  \neq k}} B_{\sigma_t( k ) \sigma_t (j)} (t) \delta_{i,\sigma_t(k)} \delta_{st} \delta_{u+1,t} - \sum_{\substack{(i,s ) \in e \\ (j,u) \in f \\ j  \neq k}} B_{kj}  (t-1) \delta_{ik} \delta_{s+1,t} \delta_{u+1,t} \\
	=&\sum_{\substack{(i,s ) \in e \\ (j,u) \in f}} -B_{k j} (t-1) \delta_{i,\sigma_t(k)} \delta_{st} \delta_{u+1,t} - \sum_{\substack{(i,s ) \in e \\ (j,u) \in f }} B_{kj}  (t-1) \delta_{ik} \delta_{s+1,t} \delta_{u+1,t} \\
	=&  \left( \sum_{(i,s ) \in e}  \delta_{i,\sigma_t(k)} \delta_{st}  +  \delta_{ik} \delta_{s+1,t} \right) \left( \sum_{ (j,u) \in f} -B_{k j} (t-1)\delta_{u+1,t}    \right)\\
	=&  \left( \sum_{(i,s ) \in e}  \delta_{i,\sigma_t(k)} \delta_{st}  +  \delta_{ik} \delta_{s+1,t} \right) \left( \sum_{ (j,u) \in f} \bigl(\maxzero{-B_{k j} (t-1)}  -\maxzero{B_{k j} (t-1)} \bigr) \delta_{u+1,t}    \right) \\
	=&n_{et}^0 n_{ft}^- - n_{et}^0 n_{ft}^+.
	\end{align*}
	Similarly, using Lemma \ref{lemma: n0} yields
	\begin{align*}
	& \sum_{\substack{(i,s ) \in e \\ (j,u) \in f \\ i  \neq k}} B_{\sigma_t (i)\sigma_t( k ) } (t)\delta_{j,\sigma_t (k) } \delta_{s+1,t} \delta_{ut} - \sum_{\substack{(i,s ) \in e \\ (j,u) \in f \\ i \neq k}} B_{ik}  (t-1) \delta_{s+1,t} \delta_{u+1,t} \\
	=&\sum_{\substack{(i,s ) \in e \\ (j,u) \in f}} B_{ki} (t-1) \delta_{j,\sigma_t (k) } \delta_{s+1,t} \delta_{u,t} + \sum_{\substack{(i,s ) \in e \\ (j,u) \in f }} B_{ki}  (t-1) \delta_{jk} \delta_{s+1,t} \delta_{u+1,t} \\
	=&   \left( \sum_{ (i,s) \in e} B_{ki} (t-1)\delta_{s+1,t}    \right) \left( \sum_{(j,u ) \in f}  \delta_{j,\sigma_t(k)} \delta_{ut}  +  \delta_{jk} \delta_{u+1,t} \right)\\
	=&   \left( \sum_{ (i,s) \in e} \bigl(-\maxzero{-B_{ki} (t-1)}  +\maxzero{B_{ki} (t-1)} \bigr) \delta_{s+1,t}    \right)\left( \sum_{(j,u) \in f}  \delta_{j,\sigma_t(k)} \delta_{ut}  +  \delta_{jk} \delta_{u+1,t} \right)  \\
	=&-n_{et}^- n_{ft}^0 + n_{et}^+ n_{ft}^0.
	\end{align*}
	This allows us to obtain
	\begin{align*}
	X(t) - X(t-1) =
	n_{et}^+ n_{ft}^0 + n_{et}^0 n_{ft}^- + n_{et}^- n_{ft}^+ 
	- n_{et}^- n_{ft}^0 - n_{et}^0 n_{ft}^+ - n_{et}^+ n_{ft}^- ,
	\end{align*}
	implying that
	\begin{align*}
	X(T) - X(0) = \sum_{t=1}^T (n_{et}^+ n_{ft}^0 + n_{et}^0 n_{ft}^- + n_{et}^- n_{ft}^+ 
	- n_{et}^- n_{ft}^0 - n_{et}^0 n_{ft}^+ - n_{et}^+ n_{ft}^-) .
	\end{align*}
	Since $B(T) = B(0)$, we have $X(T)= X(0)$.
	This completes the proof.
\end{proof}

The matrix $A_+ A_-^T$ is symmetric if and only if
\begin{align}\label{eq: symplectic orthogonal}
	\sum_{t=1}^T ( a_{et}^+ a_{ft}^- - a_{et}^- a_{ft}^+ ) = 0
\end{align}
for all $e,f \in \mathcal{E}$, i.e., rows in the $\mathcal{E} \times 2T$ matrix $(A_+ \; A_- )$ are orthogonal to each other with respect to the standard symplectic form.
If a mutation sequence is in Section \ref{section: surface examples},
the equation \eqref{eq: symplectic orthogonal} coincides with a formula obtained by Neumann-Zagier~\cite[(24)]{NeumannZagier}.

\subsection{Obtaining Neumann-Zagier matrices from block matrices}\label{subsection: NZ from X}
For all $\varepsilon \in \{ \pm ,0 \}$, let $F_{k,\varepsilon}(B)$ be the matrix that is equal to the identity matrix except that the $k$-th row is given by
\begin{align}\label{eq:def of F}
(F_{k,\varepsilon}(B))_{kj} =
\begin{cases}
-1 & \text{if $j=k$}, \\
\maxzero{\varepsilon B_{kj}} & \text{if $j \neq k$}.
\end{cases}
\end{align}
The transpose $F_{k,\varepsilon}(B)^T$ is then equal to the identity matrix except that $k$-th column is given by
\begin{align}\label{eq:def of F^T}
(F_{k,\varepsilon}(B)^T)_{ik} =
\begin{cases}
-1 & \text{if $i=k$}, \\
\maxzero{\varepsilon B_{ki}} & \text{if $i \neq k$}.
\end{cases}
\end{align}

Let $\gamma=(B,m,\sigma)$ be a mutation sequence.
For a given sign sequence $\boldsymbol{\varepsilon} =(\varepsilon_1 , \dots , \varepsilon_T) \in \{ \pm,0 \}^T$, define the block matrix $L_{\boldsymbol{\varepsilon}}$ by
\begin{align}\label{eq: def of L}
L_{\boldsymbol{\varepsilon}} =
\begin{pmatrix}
O &  \cdots & O &  P_{\sigma_1} F_{m_1,\varepsilon_1}^T \\
P_{\sigma_2} F_{m_2,\varepsilon_2}^T & \cdots & O & O \\
\vdots &  \ddots & \vdots & \vdots \\
O  &  \cdots &  P_{\sigma_T} F_{m_T,\varepsilon_T}^T &O \\
\end{pmatrix},
\end{align}
where we denote $F_{m_t,\varepsilon_t}(B(t-1))^T$ by $F_{m_t,\varepsilon_t}^T$ for brevity.
When $\varepsilon_t = +$ for all $t=1 ,\dots, T$, we denote $L_{\boldsymbol{\varepsilon}}$ by $L_+$,
defining $L_-$ and $L_0$ similarly.
In addition, let $X_{\boldsymbol{\varepsilon}} = (I - \alpha_{m,\sigma})^{-1} (I - L_{\boldsymbol{\varepsilon}})$, where $\alpha_{m,\sigma}$ is defined by \eqref{eq: def of alpha}.
Here, we assume that $\alpha_{m,\sigma}$ is nilpotent, so $ I - \alpha_{m,\sigma}$ has the inverse.
We also define $X_+$, $X_-$, and $X_0$ in the natural way.
For a $T \times T$ block matrix whose blocks are $n\times n$ matrices,
we assume that the $0$-th block means the $T$-th block and $(i,t)$ corresponds the $i$-th index in the $t$-th block, in the same manner as $\alpha_{m,\sigma}$ in Section \ref{subsection: fully mutated}.
\begin{lemma}\label{lemma: NZ from X}
	Suppose that $\gamma$ is fully mutated, and
	let $A_{\boldsymbol{\varepsilon}}=(a_{st}^{\varepsilon_t})_{1 \leq s,t \leq T}$ be its $\boldsymbol{\varepsilon}$-signed Neumann-Zagier matrix for the sign sequence $\boldsymbol{\varepsilon} =(\varepsilon_1 , \dots , \varepsilon_T) \in \{  \pm,0 \}^T$.
	Then, we have the following.
	\begin{enumerate}
		\item The matrix $X_{\boldsymbol{\varepsilon}}$ is the identity matrix except for the $m_t$-th column in the $(t-1)$-th block for $t= 1, \dots, T$.
		\item For all $1 \leq s,t \leq T$, we have $(X_{\boldsymbol{\varepsilon}})_{(m_s,s-1), (m_t,t-1)} = a_{st}^{\varepsilon_t}$.
	\end{enumerate}
\end{lemma}
\begin{proof}
	(i) follows from the facts that $X_{\boldsymbol{\varepsilon}}$ can be rewritten as
	\begin{align*}
	X_{\boldsymbol{\varepsilon}} &= (I - \alpha_{m,\sigma})^{-1} (I -L_{\boldsymbol{\varepsilon}}) \\
	&= (I - \alpha_{m,\sigma})^{-1} (I - \alpha_{m,\sigma} - ( L_{\boldsymbol{\varepsilon}} - \alpha_{m,\sigma}) ) \\
	&= I -  (I - \alpha_{m,\sigma})^{-1} ( L_{\boldsymbol{\varepsilon}} - \alpha_{m,\sigma}),
	\end{align*}
	and that the columns of $( L_{\boldsymbol{\varepsilon}} - \alpha_{m,\sigma})$ are zero except for column $(m_t , t-1)$, for $t=1 , \dots, T$.
	Now, we prove (ii).
	Let $x_{(i,s),(j,t)}$ be the entry in the $(i,s)$-th row and $(j,t)$-th column of $(I - \alpha_{m,\sigma})^{-1}$.
	Let $G_{\alpha_{m,\sigma}}$ be the graph whose adjacency matrix is $\alpha_{m,\sigma}$ as defined in Section \ref{subsection: fully mutated}.
	The entry in the $(i,s)$-th row and $(j,t)$-th column of $(\alpha_{m,\sigma})^k$ is the number of length $k$ paths in $G_{\alpha_{m,\sigma}}$ from $(i,s)$ to $(j,t)$.
	Thus $(i,s)$ and $(j,t)$ are connected in $G_{\alpha_{m,\sigma}}$ if and only if $(\alpha_{m,\sigma})^{k}_{(i,s),(j,t)} = 1$ for some non-negative integer $k$.
	Since
	\begin{align*}
	(I - \alpha_{m,\sigma})^{-1} = \sum_{ k \geq 0} (\alpha_{m,\sigma})^k
	\end{align*}
	(where the right-hand side is actually a finite sum), we have
	\begin{align}
	x_{(i,s),(j,t)} = 
	\begin{cases}
	1 & \text{if $(i,s) \sim (j,t)$}, \\
	0 & \text{otherwise}.
	\end{cases}
	\end{align}
	Let $e= [  (m_s,s-1) ] \in \mathcal{E}$ be the label of the black vertex in $\network_\gamma$ that is the equivalence class of $(m_s,s-1)$.
	Then
	\begin{align*}
	(X_{\boldsymbol{\varepsilon}})_{(m_s,s-1), (m_t,t-1)} 
	&=  \delta_{st} - \sum_{\substack{1 \leq i \leq n \\ 1 \leq u \leq T}} x_{(m_s,s-1),(i,u)} (L_{\boldsymbol{\varepsilon}} - \alpha_{m,\sigma})_{(i,u),(m_t ,t-1)} \\
	&= \delta_{st}  +  \sum_{(i,u) \in e} \delta_{i,\sigma_t (m_t)} \delta_{tu} \\
	&\quad - \sum_{(i,u) \in e} \maxzero{\varepsilon_t B_{m_t, \sigma_t^{-1} (i) } (t-1)} \delta_{tu} ,
	\end{align*}
	which is equal to the $(s,t)$-entry of $A_{\boldsymbol{\varepsilon}}$ by Lemma \ref{lemma: n0} and \ref{lemma: n epsilon}.
\end{proof}

The following result is an important consequence of Lemma \ref{lemma: NZ from X}.
\begin{theorem}\label{theorem: det formula of F}
	Let $\gamma$ be a fully mutated mutation sequence, and let 
	\begin{align*}
	F_{\boldsymbol{\varepsilon}} =  F_{m_1, \varepsilon_1} P_{\sigma_1^{-1}} F_{m_1, \varepsilon_2}  P_{\sigma_2^{-1}} \cdots F_{m_T , \varepsilon_T}  P_{\sigma_T^{-1}},
	\end{align*}
	for $\boldsymbol{\varepsilon} =(\varepsilon_1 , \dots , \varepsilon_T) \in \{  \pm,0 \}^T$.
	Then
	\begin{align}
	\det (I_n - F_{\boldsymbol{\varepsilon}} ) = \det A_{\boldsymbol{\varepsilon}}.
	\end{align}
\end{theorem}
\begin{proof}
	By Lemma \ref{lemma: NZ from X}, we obtain $\det X_{\boldsymbol{\varepsilon}} = \det A_{\boldsymbol{\varepsilon}}$.
	Since $\alpha_{m,\sigma}$ is nilpotent, we have $\det (I - \alpha_{m,\sigma}) = 1$. Thus, we find that
	\begin{align*}
	\det A_{\boldsymbol{\varepsilon}} &= \det X_{\boldsymbol{\varepsilon}} \\
	&=\det (I - \alpha_{m,\sigma})^{-1}  \det ( I - L_{\boldsymbol{\varepsilon}}  ) \\
	&= \det ( I - L_{\boldsymbol{\varepsilon}}  ) \\
	&= \det (I_n - F_{\boldsymbol{\varepsilon}}^T  ) \\
	&= \det (I_n - F_{\boldsymbol{\varepsilon}}).
	\end{align*}
\end{proof}

\section{Jacobian matrix}\label{section: jacobian matrix}
In this section, we study the Jacobian matrices associated with mutation sequences,
giving a determinant formula for them (Theorem \ref{theorem:det formula}).
For more detail, we calculate a special value of their characteristic polynomials by using the Neumann-Zagier matrices of the associated mutation sequence and the $Y$'s at the mutation points.
\subsection{Derivatives of Y-seed mutations}
The partial derivatives of $\widetilde{Y}=\mu_k(Y)$ can be calculated from \eqref{eq: Y mutation} as follows:
\begin{align}
	\frac{\partial \widetilde{Y}_i}{\partial Y_j} &=
	\begin{cases}
	-Y_{k}^{-2} & \text{if $i=j=k$}, \\
	\delta_{ij} (Y_k^{-1} + 1)^{-B_{ki}} & \text{if $j \neq k$ and $B_{ki} \geq 0$}, \\
	\delta_{ij} (Y_k + 1)^{-B_{ki}} & \text{if $j \neq k$ and $B_{ki} \leq 0$}, \\
	B_{ki} Y_{i} (Y_k^{-1} + 1)^{-B_{ki}} (Y_k + 1)^{-1} Y_k^{-1}  &
	\text{if $j=k,i\neq k, B_{ki} \geq 0$} ,\\
	-B_{ki} Y_{i} (Y_k + 1)^{-B_{ki}} (Y_k + 1)^{-1} &
	\text{if $j=k,i\neq k, B_{ki} \leq 0$},
	\end{cases} \\
	\label{eq: Y drivative}
	&=
	\begin{cases}
	-Y_{k}^{-2} & \text{if $i=j=k$}, \\
	\delta_{ij} \widetilde{Y}_{i} Y_{i}^{-1} & \text{if $j \neq k$}, \\
	B_{ki} \widetilde{Y}_{i} (Y_k+1)^{-1} Y_{k}^{-1}  &
	\text{if $j=k,i\neq k, B_{ki} \geq 0$}, \\
	-B_{ki}  \widetilde{Y}_{i} (Y_k^{-1}+1)^{-1} Y_k^{-1}  &
	\text{if $j=k,i\neq k, B_{ki} \leq 0$}.
	\end{cases}
\end{align}
Let $J_k(B)$ be the corresponding Jacobian matrix
\begin{align}
J_{k}(B) = \left( \frac{\partial \widetilde{Y}_i}{\partial Y_j}  \right)_{1 \leq i,j \leq n}.
\end{align}
The Jacobian matrix $J_{k}(B)$ can be expressed using the matrices $F_{k,+}(B)$ and $F_{k,-}(B)$ defined in Section \ref{subsection: NZ from X}, as follows.
\begin{proposition}\label{proposition: J plus minus}
	Define rational functions $z_+ , z_- \in \univsf$ by
	\begin{align*}
	z_+ =\frac{Y_k}{Y_k+1} , \;
	z_- =\frac{1}{Y_k+1},
	\end{align*}
	and diagonal matrices $\mathcal{Y}, \widetilde{\mathcal{Y}}$ by
	\begin{align*}
		\mathcal{Y} = 
		\begin{pmatrix}
		Y_1 &   &     \\
		  & \ddots  &   \\
		 &    &  Y_n 
		\end{pmatrix},\;
		\widetilde{\mathcal{Y}} = 
		\begin{pmatrix}
		\widetilde{Y}_1 &  &    \\
		   & \ddots  &   \\
		 &  &  \widetilde{Y}_n  
		\end{pmatrix}.
	\end{align*}
	Then
	\begin{align}
	J_k(B) = \widetilde{\mathcal{Y}} (z_- F_{k,+}(B)^T  + z_+ F_{k,-}(B)^T) \mathcal{Y}^{-1}.
	\end{align}
\end{proposition}
\begin{proof}
	This follows from \eqref{eq: Y drivative} and the fact that $z_+ + z_- = 1$.
\end{proof}

\subsection{Determinant formula for the Jacobian matrix}
Let $\gamma$ be a fully mutated mutation sequence, and $\mu_\gamma$ be the cluster transformation of $\gamma$, defined by \eqref{eq: def of f}.
In addition, let $J_\gamma (y)$ be the Jacobian matrix of $\mu_\gamma$:
\begin{align}
J_\gamma (y) = \left( \frac{\partial \mu_i}{\partial y_j} \right)_{1 \leq i,j \leq n},
\end{align}
where $\mu_i$ is the $i$-th component of $\mu_\gamma(y)$.
Let $\mathcal{Y}(t)$ be the $n \times n$ diagonal matrix whose $i$-th entry
is $Y_i(t)$:
\begin{align*}
\mathcal{Y}(t)=
\begin{pmatrix}
Y_1(t) &  &    \\
 & \ddots  &   \\
 &    &  Y_n(t) \\
\end{pmatrix}.
\end{align*}
Next, define the $n \times n$ matrix $K_\gamma(y)$ by
\begin{align}\label{eq: def of K}
K_\gamma(y) = \mathcal{Y}(T)^{-1} J_\gamma(y) \mathcal{Y}(0).
\end{align}
Let $Z_{+}(y)$ and $Z_{-}(y)$ be the $T \times T$ diagonal matrices whose $t$-th entries are $z_{t,+}$ and $z_{t,-}$, which were defined in \eqref{eq: def of z}:
\begin{align}\label{eq: def of Z}
Z_{+}(y) = 
\begin{pmatrix}
z_{1,+} &  &     \\
& \ddots  &   \\
&  &  z_{T,+}  \\
\end{pmatrix},\;
Z_{-}(y) = 
\begin{pmatrix}
z_{1,-} &  &     \\
& \ddots  &   \\
&  &  z_{T,-}  \\
\end{pmatrix}.
\end{align}
\begin{theorem}\label{theorem:det formula}
	Let $\gamma$ be a fully mutated mutation sequence.
	Then
	\begin{align}\label{eq:det formula}
	\det(I_n - K_\gamma(y)) = \det (A_+ Z_-(y) + A_-  Z_+(y) ).
	\end{align}
\end{theorem}
\begin{proof}
	Define the block matrix $\mathcal{L}$ by
	\begin{align*}
		\mathcal{L}= 
		\begin{pmatrix}
		O &  \cdots &O& O &  P_{\sigma_1} J_{m_1}\mathcal{Y}(0) \\
		P_{\sigma_2} J_{m_2} & \cdots & O&O & O \\
		\vdots &  \ddots & \vdots & \vdots& \vdots \\
		O & \cdots & P_{\sigma_{T-1}} J_{m_{T-1}} &O& O \\
		O  &  \cdots & O& \mathcal{Y}(T) ^{-1} P_{\sigma_T} J_{m_T} &O \\
		\end{pmatrix},
	\end{align*}
	where we have denoted the Jacobian matrix $J_{m_t}(B(t-1))$ by $J_{m_t}$.
	Let $\mathcal{Y}$ be the diagonal matrix consisting of $T$ blocks where the $t$-th block is $\mathcal{Y}(t)$ for $t=1,\dots , T-1 $ and the identity matrix $I_n$ for $t=T$:
	\begin{align*}
	\mathcal{Y}=
	\begin{pmatrix}
	\mathcal{Y}(1)  & \cdots & O &  O   \\
	\vdots  & \ddots & \vdots  &  \vdots \\
	O  & \cdots & \mathcal{Y}(T-1) & O \\
	O & \cdots & O & I_n 
	\end{pmatrix}.
	\end{align*}
	Let $\mathcal{Z}_+$ and $\mathcal{Z}_-$ be the diagonal matrix consisting of $T$ blocks whose $t$-th block is $z_{t+1,+} I_n$ and $z_{t+1,-} I_n$ for $t=1, \dots, T-1$, and $T$-th block is $z_{1,+} I_n$ and $z_{1,-} I_n$:
	\begin{align*}
	\mathcal{Z}_+=
	\begin{pmatrix}
	z_{2,+} I_n  & \cdots & O &  O   \\
	\vdots  & \ddots & \vdots  &  \vdots \\
	O  & \cdots & z_{T,+} I_n & O \\
	O & \cdots & O & z_{1,+} I_n 
	\end{pmatrix}, \;
	\mathcal{Z}_-=
	\begin{pmatrix}
	z_{2,-} I_n  & \cdots & O &  O   \\
	\vdots  & \ddots & \vdots  &  \vdots \\
	O  & \cdots & z_{T,-} I_n & O \\
	O & \cdots & O & z_{1,-} I_n 
	\end{pmatrix}.
	\end{align*}
	Because
	\begin{align*}
	J_{m_t} = P_{\sigma_t}^{-1} \mathcal{Y}(t) P_{\sigma_t} (z_{t,-} F_{m_t,+}^T  + z_{t,+} F_{m_t,-}^T )  \mathcal{Y}(t-1)^{-1}
	\end{align*}
	by Proposition \ref{proposition: J plus minus}, the matrix $\mathcal{L}$ can be expressed as follows:
	\begin{align*}
	\mathcal{L} 
	= \mathcal{Y} ( L_{+} \mathcal{Z}_- + L_{-} \mathcal{Z}_+) \mathcal{Y}^{-1},
	\end{align*}
	where $L_{\pm}$ are defined by \eqref{eq: def of L}.
	Thus, we find that
	\begin{align*}
	& \mathcal{Y} (I -  \alpha_{m,\sigma} )^{-1} \mathcal{Y}^{-1} (I - \mathcal{L}) \\
	=&  \mathcal{Y}(I -  \alpha_{m,\sigma} )^{-1}  ( I -  L_{+} \mathcal{Z}_- - L_{-} \mathcal{Z}_+ ) \mathcal{Y}^{-1} \\
	=& \mathcal{Y} (  (I -  \alpha_{m,\sigma} )^{-1}
	-  (I -  \alpha_{m,\sigma} )^{-1} L_{+} \mathcal{Z}_-  -  (I -  \alpha_{m,\sigma} )^{-1} L_{-} \mathcal{Z}_+ ) \mathcal{Y}^{-1} \\
	=& \mathcal{Y} (   (I -  \alpha_{m,\sigma} )^{-1} (I - \mathcal{Z}_- - \mathcal{Z}_+)
	+  X_{+} \mathcal{Z}_- + X_{-} \mathcal{Z}_+ ) \mathcal{Y}^{-1} \\
	=& \mathcal{Y} (
	X_{+}  \mathcal{Z}_-
	+  X_{-}  \mathcal{Z}_+ ) \mathcal{Y}^{-1}.
	\end{align*}
	Using Lemma \ref{lemma: NZ from X}, we obtain
	\begin{align*}
	\det(X_{+}  \mathcal{Z}_-
	+  X_{-}  \mathcal{Z}_+) = \det (A_+ Z_-(y) + A_-  Z_+(y) ).
	\end{align*}
	This implies that
	\begin{align*}
	\det(I - \mathcal{L}) = \det (A_+ Z_-(y) + A_-  Z_+(y) ),
	\end{align*}
	since $\det (I - \alpha_{m,\sigma}) =1$.
	
	On the other hand, the chain rule for Jacobian matrices implies that
	\begin{align*}
	J_{\gamma}(y) = P_{\sigma_T} J_{m_T} \cdots  P_{\sigma_1} J_{m_1}.
	\end{align*}
	Therefore,
	\begin{equation*}
	\begin{split}
	\det (I - \mathcal{L}) &= \det (I_n - \mathcal{Y}(T)^{-1} P_{\sigma_T} J_{m_T} \cdots  P_{\sigma_1} J_{m_1} \mathcal{Y}(0)) \\
	&= \det (I_n - \mathcal{Y}(T)^{-1} J_\gamma(y) \mathcal{Y}(0)) \\
	&=\det(I_n - K_\gamma(y)),
	\end{split}
	\end{equation*}
	and we can conclude that
	\begin{align*}
	\det (I_n - K_\gamma(y)) =   \det (A_+ Z_-(y) + A_-  Z_+(y) ).
	\end{align*}
\end{proof}

\begin{example}\label{example: main theorem}
	Set $T$ to be a fixed non-negative integer, and
	let $\gamma$ be the mutation loop defined in Example \ref{example: mutation loop}.
	By the definition of mutation networks (Section \ref{subsection: mutation network}),
	we have
	\begin{align*}
	\network_{\gamma} = 
	\bigcup_{t \in \Z / T \Z }\;
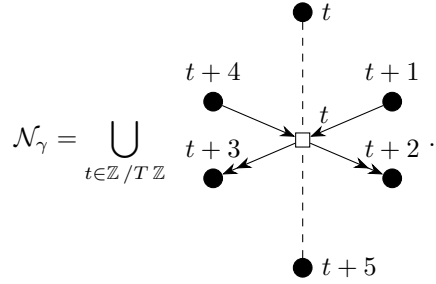
	\begin{tikzpicture}
	[baseline={([yshift=-.5ex]current bounding box.center)},scale=1.7,auto=left,black_vertex/.style={circle,draw,fill,scale=0.75},square_vertex/.style={rectangle,scale=0.8,draw}]
	\node (t) at (0,0) [square_vertex][label=above right:$t$]{};
	\node (e0) at (0,1) [black_vertex][label=right:$t$]{};
	\node (e1) at (0.7,0.3) [black_vertex][label=above:$t+1$]{};
	\node (e2) at (0.7,-0.3) [black_vertex][label=above:$t+2$]{};
	\node (e3) at (-0.7,-0.3) [black_vertex][label=above:$t+3$]{};
	\node (e4) at (-0.7,0.3) [black_vertex][label=above:$t+4$]{};
	\node (e5) at (0,-1) [black_vertex][label=right:$t+5$]{};
	\draw [dashed] (e0)edge(t);
	\draw [dashed] (e5)edge(t);
	\draw [arrows={-Stealth[scale=1.2] Stealth[scale=1.2]}] (t)edge (e2);
	\draw [arrows={-Stealth[scale=1.2] Stealth[scale=1.2]}] (t)edge (e3);
	\draw [arrows={-Stealth[scale=1.2]}] (e1)edge (t);
	\draw [arrows={-Stealth[scale=1.2]}] (e4)edge (t);
	\end{tikzpicture}.
	\end{align*}
	
	If we now assume that $T = 3$, this becomes
	\begin{align*}
	\network_{\gamma} = 
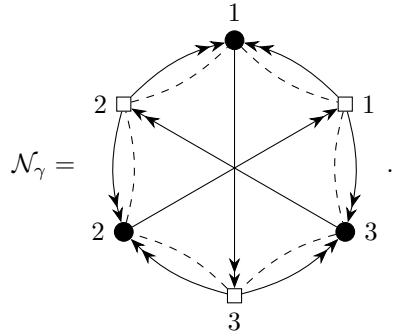
	\begin{tikzpicture}
	[baseline={([yshift=-.5ex]current bounding box.center)},scale=1.7,auto=left,black_vertex/.style={circle,draw,fill,scale=0.75},square_vertex/.style={rectangle,scale=0.8,draw}]
	\node (e1) at (90:1) [black_vertex][label=above:$1$]{};
	\node (t1) at (90-60:1) [square_vertex][label=right:$1$]{};
	\node (e3) at (90-60*2:1) [black_vertex][label=right:$3$]{};
	\node (t3) at (90-60*3:1) [square_vertex][label=below:$3$]{};
	\node (e2) at (90-60*4:1) [black_vertex][label=left:$2$]{};
	\node (t2) at (90-60*5:1) [square_vertex][label=left:$2$]{};
	\draw [dashed] (e1)edge [bend right=15](t1);
	\draw [dashed] (t1)edge [bend right=15](e3);
	\draw [dashed] (e3)edge [bend right=15](t3);
	\draw [dashed] (t3)edge [bend right=15](e2);
	\draw [dashed] (e2)edge [bend right=15](t2);
	\draw [dashed] (t2)edge [bend right=15](e1);
	\draw [arrows={-Stealth[scale=1.2] Stealth[scale=1.2]}] (t1)edge [bend right=15](e1);
	\draw [arrows={-Stealth[scale=1.2] Stealth[scale=1.2]}] (t1)edge [bend left=15](e3);
	\draw [arrows={-Stealth[scale=1.2] Stealth[scale=1.2]}] (t3)edge [bend right=15](e3);
	\draw [arrows={-Stealth[scale=1.2] Stealth[scale=1.2]}] (t3)edge [bend left=15](e2);
	\draw [arrows={-Stealth[scale=1.2] Stealth[scale=1.2]}] (t2)edge [bend right=15](e2);
	\draw [arrows={-Stealth[scale=1.2] Stealth[scale=1.2]}] (t2)edge [bend left=15](e1);
	\draw [arrows={-Stealth[scale=1.2] Stealth[scale=1.2]}] (e2)edge (t1);
	\draw [arrows={-Stealth[scale=1.2] Stealth[scale=1.2]}] (e3)edge (t2);
	\draw [arrows={-Stealth[scale=1.2] Stealth[scale=1.2]}] (e1)edge (t3);
	\end{tikzpicture}.
	\end{align*}
The Neumann-Zagier matrices can be obtained from the adjacency matrices of $\network_\gamma$ by \eqref{eq: def of A+} and \eqref{eq: def of A-}, yielding
\begin{align*}
	A_+ &= N_0 - N_+ \\
	&= 
	\begin{pmatrix}
	1 & 1 & 0 \\
	0 & 1 & 1 \\
	1 & 0 & 1
	\end{pmatrix} -
	\begin{pmatrix}
	2 & 2 & 0 \\
	0 & 2 & 2 \\
	2 & 0 & 2
	\end{pmatrix} \\
	&=
	\begin{pmatrix}
	-1 & -1 & 0 \\
	0 & -1 & -1 \\
	-1 & 0 & -1
	\end{pmatrix}, \\
	A_- &= N_0 - N_- \\
	&= 
	\begin{pmatrix}
	1 & 1 & 0 \\
	0 & 1 & 1 \\
	1 & 0 & 1
	\end{pmatrix} -
	\begin{pmatrix}
	0 & 0 & 2 \\
	2 & 0 & 0 \\
	0 & 2 & 0
	\end{pmatrix} \\
	&=
	\begin{pmatrix}
	1 & 1 & -2 \\
	-2 & 1 & 1 \\
	1 & -2 & 1
	\end{pmatrix}.
\end{align*}
The matrices $Z_+(y)$ and $Z_-(y)$ (defined by \eqref{eq: def of Z}) are as follows:
\begin{align*}
&Z_+ (y)= 
\begin{pmatrix}
\frac{ y_1}{ y_1 + 1} & 0 & 0 \\
0 & \frac{ y_2  (y_1 + 1)}{y_1 y_2 + y_2 + 1} & 0\\
0 & 0 &  \frac{ y_1^{2}y_3  (y_1 y_2 + y_2 + 1)}{y_1^{3} y_2 y_3 + y_1^{2} y_2 y_3 + y_1^{2} y_3 + y_1^{2} + 2 y_1 + 1}
\end{pmatrix},\\
&Z_-(y) = 
\begin{pmatrix}
\frac{1}{y_1 + 1} & 0 & 0 \\
0 &\frac{1}{y_1 y_2 + y_2 + 1} & 0\\
0 & 0 &  \frac{ (y_1 +1)^{2} }{y_1^{3} y_2 y_3 + y_1^{2} y_2 y_3 + y_1^{2} y_3 + y_1^{2} + 2 y_1 + 1}
\end{pmatrix} .
\end{align*}
Theorem \ref{theorem:det formula} now gives
\begin{align*}
&\quad \det (I_5 - K_\gamma (y))  \\
&= \det(A_+ Z_-(y) + A_- Z_+(y)) \\
&= -2   (y_1 + 1)^{-1}   (y_1 y_2 + y_2 + 1)^{-1}   (y_1^3 y_2 y_3 + y_1^2 y_2 y_3 + y_1^2 y_3 + y_1^2 + 2 y_1 + 1)^{-1}   \\
&\quad \cdot (3 y_1^4 y_2^2 y_3 + 3 y_1^4 y_2 y_3 + 6 y_1^3 y_2^2 y_3 + 3 y_1^4 y_2 + 4 y_1^3 y_2 y_3  + 3 y_1^2 y_2^2 y_3 + 7 y_1^3 y_2 \\ 
&\quad\quad + 3 y_1^3 y_3 + y_1^2 y_2 y_3 - 2 y_1^3 + 3 y_1^2 y_2 - 2 y_1^2 y_3 - 3 y_1^2 - 3 y_1 y_2 - 2 y_2 + 1).
\end{align*}
\end{example}

For a fully mutated mutation sequence $\gamma$,
we denote the right-hand side of \eqref{eq:det formula} by $\tau_{\gamma} (y)$.
Since the left-hand side of \eqref{eq:det formula} depends only on the cluster transformation, we obtain the following:
\begin{corollary}\label{corollary: tau = tau'}
	Suppose that $\gamma$ and $\gamma'$ are fully mutated mutation sequences.
	Then $\mu_\gamma(y) = \mu_{\gamma'}(y)$ implies $\tau_{\gamma}(y) = \tau_{\gamma'}(y)$.
\end{corollary}

\begin{example}
\label{example: tau=tau'}
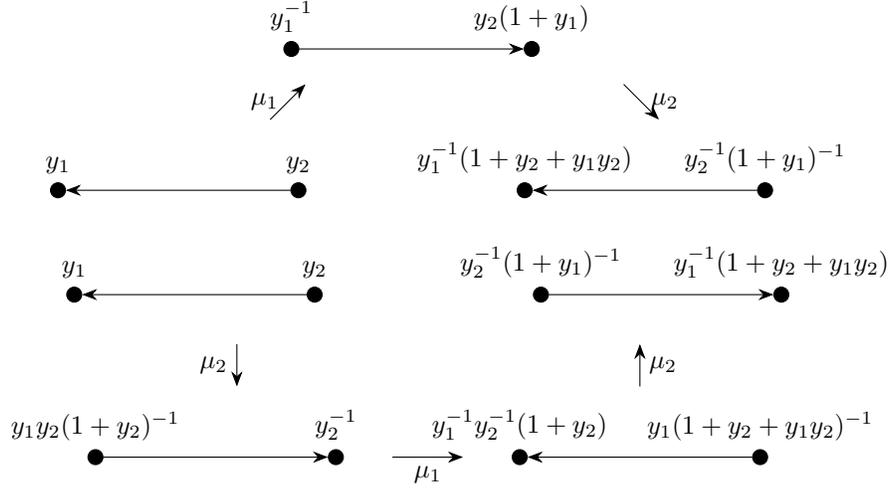
\begin{figure}
		\begin{tikzpicture}[scale=0.94]
		\node (1) at (0,0)[circle,scale=0.6,draw,fill] [label=above:$y_1$] {};
		\node (2) at (3+0.4,0) [circle,scale=0.6,draw,fill] [label=above:$y_2$] {};
		\draw[arrows={-Stealth[scale=1.2]}] (2)--(1) ;
		\draw[arrows={-Stealth[scale=1.2]}] (3,1.0)--(3.5,1.5) node[midway,left]{$\mu_1$};
		\node (1') at (3.5-0.2,2)[circle,scale=0.6,draw,fill] [label=above:$y_1^{-1}$]{};
		\node (2') at (6.5+0.2,2)[circle,scale=0.6,draw,fill] [label=above:$y_2(1+y_1)$] {};
		\draw[arrows={-Stealth[scale=1.2]}] (1')--(2');
		\draw[arrows={-Stealth[scale=1.2]}] (8,1.5)--(8.5,1.0) node[midway,right]{$\mu_2$};
		\node (1'') at (7-0.4,0) [circle,scale=0.6,draw,fill] [label=above:$y_1^{-1}(1+y_2+y_1 y_2)$] {};
		\node (2'') at (10,0) [circle,scale=0.6,draw,fill] [label=above:$y_2^{-1}(1+y_1)^{-1}$] {};
		\draw[arrows={-Stealth[scale=1.2]}] (2'')--(1'');
		\end{tikzpicture}
		\vspace{5mm}
		\\
		\begin{tikzpicture}[scale=0.94]
		\node (1) at (0,0)[circle,scale=0.6,draw,fill] [label=above:$y_1$] {};
		\node (2) at (3+0.4,0) [circle,scale=0.6,draw,fill] [label=above:$y_2$] {};
		\draw[arrows={-Stealth[scale=1.2]}] (2)--(1);
		\draw[arrows={-Stealth[scale=1.2]}] (2.3,-1+0.3)--(2.3,-1-0.3)node[midway,left]{$\mu_2$};
		\node (1') at (0.3,-2.3)[circle,scale=0.6,draw,fill] [label=above:$y_1y_2(1+y_2)^{-1}$]{};
		\node (2') at (0.2+3.5,-2.3)[circle,scale=0.6,draw,fill] [label=above:$y_2^{-1}$] {};
		\draw[arrows={-Stealth[scale=1.2]}] (1')--(2');
		\draw[arrows={-Stealth[scale=1.2]}] (8,-1-0.3)--(8,-1+0.3) node[midway,right]{$\mu_2$};
		\node (1'') at (7-0.7,-2.3)[circle,scale=0.6,draw,fill] [label=above:$y_1^{-1}y_2^{-1}(1+y_2)$]{};
		\node (2'') at (10-0.3,-2.3)[circle,scale=0.6,draw,fill] [label=above:$y_1(1+y_2+y_1y_2)^{-1}$] {};
		\draw[arrows={-Stealth[scale=1.2]}] (2'')--(1'');
		\node (1''') at (7-0.4,0) [circle,scale=0.6,draw,fill] [label=above:$y_2^{-1}(1+y_{1})^{-1}$] {};
		\node (2''') at (10,0) [circle,scale=0.6,draw,fill] [label=above:$y_1^{-1}(1+y_2+y_1y_2)$] {};
		\draw[arrows={-Stealth[scale=1.2]}] (1''')--(2''');
		\draw[arrows={-Stealth[scale=1.2]}] (4.5,-2.3)--(5.5,-2.3) node[midway,below]{$\mu_1$};
		\end{tikzpicture}
		\caption{Two mutation loops $\gamma = (B,(1,2),(\id,\id))$ and $\gamma'=(B,(2,1,2),(\id,\id, (1 \; 2)))$ yield the same cluster transformations.}
		\label{figure: gamma and gamma'}
\end{figure}
Let $B$ be the matrix defined by
\begin{align*}
B=
\begin{pmatrix}
0 & -1 \\
1 & 0
\end{pmatrix},
\end{align*}
and define two mutation loops by
$\gamma = (B, (1,2) , (\id,\id))$ and $\gamma' =(B, (2,1,2), (\id, \id, (1 \; 2))$.
Then we can see that the cluster transformations of $\gamma$ and $\gamma'$ coincide (see Figure \ref{figure: gamma and gamma'}):
\begin{align*}
	\mu_{\gamma}(y) = \mu_{\gamma'}(y) = ( y_1^{-1}(1+y_2+y_1 y_2) , y_2^{-1}(1+y_1)^{-1} ).
\end{align*}
The mutation networks of these mutation loops are given by
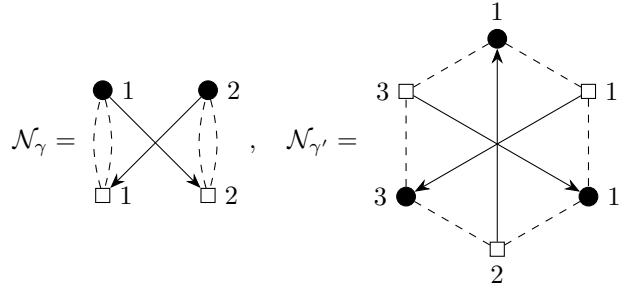
\begin{align*}
\mathcal{N}_\gamma=
\;
\begin{tikzpicture}
[baseline={([yshift=-.5ex]current bounding box.center)},scale=1.4,auto=left,black_vertex/.style={circle,draw,fill,scale=0.75},square_vertex/.style={rectangle,scale=0.8,draw}]
\node (a) at (0,0) [black_vertex][label=right:$1$]{};
\node (b) at (1,0) [black_vertex][label=right:$2$]{}; 
\node (t1) at (0,-1) [square_vertex][label=right:$1$]{};
\node (t2) at (1,-1) [square_vertex][label=right:$2$]{};
\draw [dashed] (a)edge [bend left=15](t1);
\draw [dashed] (t1)edge [bend left=15](a);
\draw [dashed] (b)edge [bend left=15](t2);
\draw [dashed] (t2)edge [bend left=15](b);
\draw[arrows={-Stealth[scale=1.2]}] (b)--(t1);
\draw[arrows={-Stealth[scale=1.2]}] (a)--(t2);
\end{tikzpicture},\quad
\mathcal{N}_{\gamma'} =
\begin{tikzpicture}
[baseline={([yshift=-.5ex]current bounding box.center)},scale=1.4,auto=left,black_vertex/.style={circle,draw,fill,scale=0.75},square_vertex/.style={rectangle,scale=0.8,draw}]
\node (e1) at (90:1) [black_vertex][label=above:$1$]{};
\node (t1) at (90-60:1) [square_vertex][label=right:$1$]{};
\node (e2) at (90-60*2:1) [black_vertex][label=right:$1$]{};
\node (t2) at (90-60*3:1) [square_vertex][label=below:$2$]{};
\node (e3) at (90-60*4:1) [black_vertex][label=left:$3$]{};
\node (t3) at (90-60*5:1) [square_vertex][label=left:$3$]{};
\draw [dashed] (e1)edge [](t1);
\draw [dashed] (t1)edge [](e2);
\draw [dashed] (e2)edge [](t2);
\draw [dashed] (t2)edge [](e3);
\draw [dashed] (e3)edge [](t3);
\draw [dashed] (t3)edge [](e1);
\draw [arrows={-Stealth[scale=1.2]}] (t2)edge (e1);
\draw [arrows={-Stealth[scale=1.2]}] (t3)edge (e2);
\draw [arrows={-Stealth[scale=1.2]}] (t1)edge (e3);
\end{tikzpicture}.
\end{align*}
Thus we have
\begin{align*}
	\tau_{\gamma}(y) 
	&= \det \left(
	\begin{pmatrix}
	2&0\\
	0&2
	\end{pmatrix}Z_- (y) +
	\begin{pmatrix}
	2&-1\\
	-1&2
	\end{pmatrix}Z_+ (y) \right) \\
	&=\det
	\begin{pmatrix}
	2 & -\frac{y_1}{1+y_1} \\
	-\frac{y_2(1+y_1)}{1+y_2+y_1 y_2} & 2
	\end{pmatrix}
\end{align*}
and
\begin{align*}
\tau_{\gamma'}(y) &= \det \left(
\begin{pmatrix}
1&-1&1\\
1&1&-1\\
-1&1&1
\end{pmatrix} Z'_- (y) +
\begin{pmatrix}
1&0&1\\
1&1&0\\
0&1&1
\end{pmatrix} Z'_+ (y) \right) \\
	&= \det
	\begin{pmatrix}
	1 & -\frac{1 + y_2}{ 1+y_2 + y_1 y_2 } & 1\\
	1 & 1 & -\frac{1+y_2 + y_1 y_2}{(1+y_1)(1+y_2)} \\
	-\frac{1}{1+y_2} & 1 & 1
	\end{pmatrix},
\end{align*}
where $Z_{\pm}(y)$ and $Z'_{\pm}(y)$ are the matrices in \eqref{eq: def of Z} for $\gamma$ and $\gamma'$, respectively.
Thus we obtain
\begin{align*}
\det
\begin{pmatrix}
2 & -\frac{y_1}{1+y_1} \\
-\frac{y_2(1+y_1)}{1+y_2+y_1 y_2} & 2
\end{pmatrix} 
= \det
\begin{pmatrix}
1 & -\frac{1 + y_2}{ 1+y_2 + y_1 y_2 } & 1\\
1 & 1 & -\frac{1+y_2 + y_1 y_2}{(1+y_1)(1+y_2)} \\
-\frac{1}{1+y_2} & 1 & 1
\end{pmatrix}
\end{align*}
by Corollary \ref{corollary: tau = tau'}.
In fact, we can see that both sides are equal to 
\begin{align*}
	 \frac{4+4y_2 +3 y_1 y_2}{1+y_2 + y_1 y_2} .
\end{align*}
\end{example}

\begin{remark}
	Theorem \ref{theorem:det formula} is purely presented in the language of cluster algebras, but their background is in three-dimensional topology.
	If a mutation sequence is what we considered in Section \ref{section: surface examples} and therefore is related to a 3-manifold,
	the left-hand side of \ref{eq:det formula} is related to Kitayama-Terashima's formula for Reidemeister torsion~\cite{KitayamaTerashima},
	while the right-hand side is related to Dimofte-Garoufalidis's conjectural formula for Reidemeister torsion~\cite{DimofteGaroufalidis}.
	We will discuss the more precise relationship in the future.
\end{remark}

\section{The $y \to 0$ limit in the determinant formula}\label{section: y->0 limit}
In this section, we describe the relation between Theorem \ref{theorem:det formula} and Y-seed mutations in the tropical semifields.
We show that a determinant formula for $C$-matrices, matrices describing Y-seeds in the tropical semifields, are obtained by taking ``$y \to 0$ limit'' in \eqref{eq:det formula}.
\subsection{Y-seed mutations in tropical semifields and  $\mathbf{c}$-vectors}
A \emph{semifield} is a multiplicative abelian group endowed with an (auxiliary) addition $\oplus$ that is commutative, associative, and distributive with respect to the multiplication.
The set $\univsf$ of subtraction-free rational expressions in the variables $y_1 ,\dots, y_n$ over $\Q$ used in the previous sections is a semifield with respect to the usual multiplication and the addition.
The set $\univsf$ is called an \emph{universal semifield}.

Another important semifields that we consider is tropical semifields.
\begin{definition}
	Let $\tropsf$ be the set of Laurent monomials in the variables $y_1 ,\dots , y_n$ equipped with the usual multiplication and the tropical addition $\oplus$ defined by
	\begin{align}
	\prod_{i=1}^n y_i^{a_i} \oplus \prod_{i=1}^n y_i^{b_i} = \prod_{i=1}^n y_i^{\min(a_i, b_i)}.
	\end{align}
	The set $\tropsf$ with these operators is called an \emph{tropical semifield}.
\end{definition}

We say that a pair $(B,Y)$ is a Y-seed in the tropical semifield if $B$ is an $n\times n$ skew-symmetrizable matrix and $Y \in \tropsf^n$.
We define Y-seed mutations in the tropical semifield by $\mu_k (B,Y) = (\widetilde{B}, \widetilde{Y})$ where $\widetilde{B}$ is the same as \eqref{eq: B mutation}, and
$\widetilde{Y}$ is defined by replacing $+$ in \eqref{eq: Y mutation} with $\oplus$:
\begin{align}
	\label{eq: Y mutation tropical}
\widetilde{Y}_i =
&\begin{cases}
Y_{k}^{-1} & \text{if $i=k$}, \\
Y_{i}  (Y_k^{-1} \oplus 1)^{-B_{ki}} & \text{if $i \neq k$, $B_{ki} \geq 0$}, \\
Y_{i}  (Y_k \oplus 1)^{-B_{ki}} & \text{if $i \neq k$, $B_{ki} \leq 0$}.
\end{cases} 
\end{align}

Let $\gamma=(B,m,\sigma)$ be a mutation sequence.
Then we have the following transitions of Y-seeds in the tropical semifield:
\begin{align}\label{intro:Y transition tropical}
(B(0),Y(0)) \xrightarrow{\sigma_1 \circ \mu_{m_1}} (B(1),Y(1)) \xrightarrow{\sigma_2 \circ \mu_{m_2}} \cdots
\xrightarrow{\sigma_T \circ \mu_{m_T}} (B(T),Y(T))
\end{align}
where $B(0)=B$ and $Y(0) = (y_1 , \dots , y_n)$. 
Now we have $Y(t) \in \tropsf^n$ for each $t=0,\dots ,T$.
Since these are Laurent monomials in $y_1 , \dots , y_n$,
we can define the integers $c_{ij} (t)$ for each $ 1 \leq i,j \leq n$ and $0 \leq t \leq T$ by the following expression:
\begin{align}
	Y_j(t) =  y_1^{c_{1j}(t)} \cdots y_n^{c_{nj}(t)}.
\end{align}
The vector 
\begin{align}
\mathbf{c}_j(t) = 
\begin{pmatrix}
	c_{1j}(t) \\
	\vdots \\
	c_{nj}(t)
\end{pmatrix} 
\end{align}
is called a (permuted) \emph{$\mathbf{c}$-vector}, and the matrix
\begin{equation}
\begin{split}
	C(t) &= ( \mathbf{c}_1(t)  \cdots  \mathbf{c}_n(t)  ) \\
	&= (c_{ij}(t))_{1 \leq i,j \leq n}
\end{split}
\end{equation}
is called a (permuted) \emph{$C$-matrix}.

It was conjectured in~\cite{FZ4} that $\mathbf{c}$-vectors has a special property called sign-coherence.
This was proved in~\cite{DWZ} when the exchange matrix $B$ is skew-symmetric, and proved in~\cite{GHKK} for general case.
We state this property in the form that we will use in this paper.
\begin{theorem}[\cite{GHKK}]\label{theorem: sign coherence}
	For any mutation sequence $\gamma$,
	the vector $\mathbf{c}_j(t)$ is non-zero and
	the entries of $ \mathbf{c}_j(t)$ are either all non-negative or all non-positive for any $j$ and $t$.
\end{theorem}

For each $t=1 , \dots , T$, let $\varepsilon_t$ be the sign of $\mathbf{c}_{m_t} (t-1)$.
This is well-defined by Theorem \ref{theorem: sign coherence}.
It follows from Proposition 1.3 in~\cite{NakanishiZelevinsky} that the $C$-matrices satisfy the following relations:
\begin{align}
	C(t+1) = C(t) F_{m_t,\varepsilon_t} P_{\sigma_t^{-1}}.
\end{align}
Thus we obtain
\begin{align}
	C(T) = F_{m_1,\varepsilon_1} P_{\sigma_1^{-1}} F_{m_2,\varepsilon_2} P_{\sigma_2^{-1}} \cdots F_{m_T,\varepsilon_T} P_{\sigma_T^{-1}}.
\end{align} 

\subsection{The $y \to 0$ limit}
Let $\epsilon$ be a positive real number.
We consider a limit by substituting $\epsilon$ for each $y_i$ and taking the limit $\epsilon \to 0$.
We denote this limit by $y \to 0$.
Let $\boldsymbol{\varepsilon}_{\mathrm{trop}} = (\varepsilon_1 , \dots , \varepsilon_T)$ where $\varepsilon_t$ is the sign of the $\mathbf{c}$-vector $\mathbf{c}_{m_t} (t-1)$.
The special case of Theorem \ref{theorem: det formula of F} for this sign sequence is obtained as the $y \to 0$ limit of Theorem \ref{theorem:det formula}.
\begin{theorem}\label{theorem: y0 limit}
	Let $\gamma$ be a fully mutated mutation sequence.
	For the matrices in \eqref{eq:det formula}, we have
	\begin{align*}
		&\lim_{y \to 0} K_\gamma(y)^T = C(T), \\
		&\lim_{y \to 0} (A_+ Z_-(y) + A_- Z_+(y)) = A_{\boldsymbol{\varepsilon}_{\mathrm{trop}}}.
	\end{align*}
	Thus the limit $y \to 0$ in \eqref{eq:det formula} yields
	\begin{align*}
		\det (I_n - C(T)) = \det A_{\boldsymbol{\varepsilon}_{\mathrm{trop}}}.
	\end{align*}
\end{theorem}
\begin{proof}
By Proposition \ref{proposition: J plus minus}, the matrix $K_\gamma(y)$ defined by \eqref{eq: def of K} has the following expression:
\begin{align*}
	K_\gamma(y) = P_{\sigma_T} (z_{T,-} F_{m_T,+}^T +z_{T,+} F_{m_T,-}^T  ) \cdots P_{\sigma_1} (z_{1,-} F_{m_1,+}^T +z_{1,+} F_{m_1,-}^T  ).
\end{align*}

By Proposition 3.13 in~\cite{FZ4}, the rational function $Y_j (t) \in \univsf $ in the transitions \eqref{eq: Y transition} can be written in the form
\begin{align*}
	Y_j(t) =  y_1^{c_{1j}(t)} \cdots y_n^{c_{nj}(t)} \prod_{i=1}^n F_i(t)(y_1 , \dots , y_n)^{B_{ij}(t)}
\end{align*}
for some polynomials $F_1(t) , \dots , F_n(t)$.
Moreover, Proposition 5.6 in ~\cite{FZ4} says that sign coherence of $\mathbf{c}$-vectors implies that these polynomials have constant term $1$.
Thus we have
\begin{align}\label{eq: Y is 0 or infty}
\lim_{y \to 0} Y_{m_t}(t) =
\begin{cases}
	0 & \text{if $\varepsilon_t = +$}, \\
	\infty & \text{if $\varepsilon_t = -$}.
\end{cases}
\end{align}
By the definition of $z_{+,t}$ and $z_{-,t}$ in \eqref{eq: def of z},
we obtain
\begin{equation}\label{eq: z is 1 or 0}
\begin{split}
&\lim_{y \to 0} z_{+,t} =
\begin{cases}
0 & \text{if $\varepsilon_t =+$}, \\
1 & \text{if $\varepsilon_t =-$},
\end{cases}\\
&\lim_{y \to 0} z_{-,t} =
\begin{cases}
1 & \text{if $\varepsilon_t =+$}, \\
0 & \text{if $\varepsilon_t =-$}.
\end{cases}
\end{split}
\end{equation}
This implies that
\begin{align*}
	\lim_{y \to 0} K_\gamma(y)^T = C(T).
\end{align*}
Using \eqref{eq: z is 1 or 0} again, we also have
\begin{align*}
	\lim_{y \to 0} (A_+ Z_-(y) + A_- Z_+(y)) = A_{\boldsymbol{\varepsilon}_{\mathrm{trop}}}.
\end{align*}
\end{proof}

We should be aware that the $y \to 0$ limit and \eqref{eq: z is 1 or 0} is used in the proof of the dilogarithm identities associated with periodic Y-systems by Nakanishi~\cite{Nakanishi}, and originally considered by Chapoton~\cite{Chapoton} in a special case.

\subsection{Non-existence of reddening sequences on once-punctured surfaces}
We assume that $B$ is a skew-symmetric matrix.
A mutation sequence $\gamma=(B,m,\sigma)$ is \emph{reddening} if all entries in the last $C$-matrix $C(T)$ are non-positive.
A mutation sequence is \emph{maximal green} if it is reddening and any sign in
$\boldsymbol{\varepsilon}_{\mathrm{trop}}$ is a plus.
The following is a fundamental property of reddening sequences.
\begin{proposition}\cite[Proposition 2.10]{BDP}
	\label{proposition: reddening}
	Suppose that $\gamma = (B,m,\sigma)$ is a reddening sequence.
	Then there exists a permutation $\nu$ of $\{ 1, \dots, n \}$ such that
	$\gamma'=(B,m,(\sigma_1 ,\dots, \sigma_{T-1},\nu \circ \sigma_T))$
	is a mutation loop and $C(T)=-I_n$.
\end{proposition}

Let $S$ be a closed oriented surfaces and $P$ be a one-point set whose element is an point in $S$.
For any ideal triangulation $\Gamma$ of $(S,P)$,
let $B_\Gamma$ be the skew-symmetric matrix defined in Section \ref{section: surface examples}.
The following result (for maximal green sequences) was proved by Ladkani~\cite{ladkani2013cluster}.
We give another proof for this result.
\begin{theorem}\label{theorem: no reddening}
	For any ideal triangulation $\Gamma$ of once-punctured surface $(S,P)$,
	the skew-symmetric matrix $B_\Gamma$ does not have reddening or maximal green sequences.
\end{theorem}
\begin{proof}
	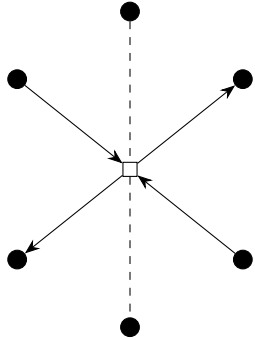
\begin{figure}
		\begin{tikzpicture}
		[baseline={([yshift=-.5ex]current bounding box.center)},scale=1.5,auto=left,black_vertex/.style={circle,draw,fill,scale=0.75},square_vertex/.style={rectangle,scale=0.8,draw}]
		\node (a) at (1,0.8) [black_vertex]{};
		\node (b) at (-1,0.8) [black_vertex]{}; 
		\node (c) at (1,-0.8) [black_vertex]{};
		\node (d) at (-1,-0.8) [black_vertex]{}; 
		\node (h) at (0,1.4) [black_vertex]{};
		\node (hp) at (0,-1.4) [black_vertex]{}; 
		\node (t) at (0,0) [square_vertex]{};
		\draw [dashed] (h) edge (t);
		\draw [dashed] (t) edge (hp);
		\draw[arrows={-Stealth[scale=1.2]}] (c)--(t);
		\draw[arrows={-Stealth[scale=1.2]}] (b)--(t);
		\draw[arrows={-Stealth[scale=1.2]}] (t)--(a);
		\draw[arrows={-Stealth[scale=1.2]}] (t)--(d);
		\end{tikzpicture}
		\caption{A square vertex of a mutation network obtained from a mutation sequence on a skew-symmetric matrix associated with a once-punctured surface.}
		\label{fig: once-puncutred}
	\end{figure}
	It is enough to show that $B_\Gamma$ does not have reddening sequences.
	Suppose that $\gamma$ is a reddening sequence on $B_\Gamma$.
	We can assume that $C(T)=-I_n$ by Proposition \ref{proposition: reddening}.
	We can also assume that $\gamma$ is fully mutated by inserting two consecutive mutations if necessary.
	Then we have
	\begin{align*}
		\det A_{\boldsymbol{\varepsilon}_{\mathrm{trop}}} = 2^n
	\end{align*}
	by Theorem \ref{theorem: y0 limit}.
	
	On the other hand, we have
	\begin{align*}
	\sum_{i=1}^n \maxzero{(B_\Gamma)_{m_t,i}(t-1) } =\sum_{i=1}^n \maxzero{-(B_\Gamma)_{m_t,i}(t-1) }  =2
	\end{align*}
	for all $t=1 , \dots, T$ since any ideal triangulation on once-punctured surfaces does not have self-folded edges.
	Thus we can see that two arrows come in and two arrows come out at any square vertex in the mutation network $\network_\gamma$ (see Figure \ref{fig: once-puncutred}).
	Since the entries in each column of $A_{\pm}$ sum to $0$ in this case,
	we obtain
	\begin{align*}
		\det A_{\boldsymbol{\varepsilon}} = 0
	\end{align*}
	for any sign sequence $\boldsymbol{\varepsilon} \in \{ \pm 1 \}^T$, which is a contradiction.
\end{proof}

\providecommand{\bysame}{\leavevmode\hbox to3em{\hrulefill}\thinspace}
\providecommand{\MR}{\relax\ifhmode\unskip\space\fi MR }
\providecommand{\MRhref}[2]{%
	\href{http://www.ams.org/mathscinet-getitem?mr=#1}{#2}
}
\providecommand{\href}[2]{#2}


\begin{thebibliography}{GHKK18}
	\bibitem[BDP14]{BDP}
	T.~Br\"{u}stle, G.~Dupont, and M.~P\'{e}rotin, \emph{On
		maximal green sequences}, Int. Math. Res. Not. IMRN (2014), no.~16,
	4547--4586. \MR{3250044}
	
	\bibitem[Cha05]{Chapoton}
	F.~Chapoton, \emph{Functional identities for the {R}ogers dilogarithm
		associated to cluster {$Y$}-systems}, Bull. London Math. Soc. \textbf{37}
	(2005), no.~5, 755--760. \MR{2164838}
	
	\bibitem[DG13]{DimofteGaroufalidis}
	T.~Dimofte and S.~Garoufalidis, \emph{The quantum content of the gluing
		equations}, Geom. Topol. \textbf{17} (2013), no.~3, 1253--1315. \MR{3073925}
	
	\bibitem[DK19]{keller2019survey}
	L.~Demonet and B.~Keller, \emph{A survey on maximal green
		sequences}, arXiv preprint arXiv:1904.09247 (2019).
	
	\bibitem[DWZ10]{DWZ}
	H.~Derksen, J.~Weyman, and A.~Zelevinsky, \emph{Quivers with potentials and
		their representations {II}: applications to cluster algebras}, J. Amer. Math.
	Soc. \textbf{23} (2010), no.~3, 749--790. \MR{2629987}
	
	\bibitem[FM11]{FordyMarsh}
	A.~P. Fordy and R.~J. Marsh, \emph{Cluster mutation-periodic quivers and
		associated {L}aurent sequences}, J. Algebraic Combin. \textbf{34} (2011),
	no.~1, 19--66. \MR{2805200}
	
	\bibitem[FZ02]{FZ1}
	S.~Fomin and A.~Zelevinsky, \emph{Cluster algebras. {I}. {F}oundations}, J.
	Amer. Math. Soc. \textbf{15} (2002), no.~2, 497--529. \MR{1887642}
	
	\bibitem[FZ07]{FZ4}
	\bysame, \emph{Cluster algebras. {IV}. {C}oefficients}, Compos. Math.
	\textbf{143} (2007), no.~1, 112--164. \MR{2295199}
	
	\bibitem[GHKK18]{GHKK}
	M.~Gross, P.~Hacking, S.~Keel, and M.~Kontsevich, \emph{Canonical bases for
		cluster algebras}, J. Amer. Math. Soc. \textbf{31} (2018), no.~2, 497--608.
	\MR{3758151}
	
	\bibitem[Kel11]{Keller2011}
	B.~Keller, \emph{On cluster theory and quantum dilogarithm identities},
	Representations of algebras and related topics, EMS Ser. Congr. Rep., Eur.
	Math. Soc., Z\"{u}rich, 2011, pp.~85--116. \MR{2931896}
	
	\bibitem[KNS11]{KNS2011}
	A.~Kuniba, T.~Nakanishi, and J.~Suzuki, \emph{{$T$}-systems and {$Y$}-systems
		in integrable systems}, J. Phys. A \textbf{44} (2011), no.~10, 103001, 146.
	\MR{2773889}
	
	\bibitem[KT15]{KitayamaTerashima}
	T.~Kitayama and Y.~Terashima, \emph{Torsion functions on moduli spaces in view
		of the cluster algebra}, Geom. Dedicata \textbf{175} (2015), 125--143.
	\MR{3323633}
	
	\bibitem[Nak11]{Nakanishi}
	T.~Nakanishi, \emph{Periodicities in cluster algebras and dilogarithm
		identities}, Representations of algebras and related topics, EMS Ser. Congr.
	Rep., Eur. Math. Soc., Z\"urich, 2011, pp.~407--443. \MR{2931902}
	
	\bibitem[NTY17]{NagaoTerashimaYamazaki}
	K.~Nagao, Y.~Terashima, and M.~Yamazaki, \emph{Hyperbolic
		3-manifolds and cluster algebras}, Nagoya Mathematical Journal (2017),
	1–25.
	
	\bibitem[NZ85]{NeumannZagier}
	W.~D. Neumann and D.~Zagier, \emph{Volumes of hyperbolic three-manifolds},
	Topology \textbf{24} (1985), no.~3, 307--332. \MR{815482}
	
	\bibitem[NZ12]{NakanishiZelevinsky}
	T.~Nakanishi and A.~Zelevinsky, \emph{On tropical dualities in cluster
		algebras}, Algebraic groups and quantum groups, Contemp. Math., vol. 565,
	Amer. Math. Soc., Providence, RI, 2012, pp.~217--226. \MR{2932428}
	
	\bibitem[Lad13]{ladkani2013cluster}
	S.~Ladkani, \emph{On cluster algebras from once punctured closed surfaces},
	arXiv preprint arXiv:1310.4454 (2013).
	
	\bibitem[Thu80]{Thurston}
	W.~P. Thurston, \emph{The geometry and topology of three-manifolds}, 1980,
	Princeton lecture notes, http://library.msri.org/books/gt3m/.
	
	\bibitem[TY14]{TerashimaYamazaki}
	Y.~Terashima and M.~Yamazaki, \emph{$\mathcal{N}=2$ theories from cluster
		algebras}, PTEP \textbf{2014} (2014), no.~2, 023B01.
	
	\bibitem[VZ11]{MashaSander}
	M.~Vlasenko and S.~Zwegers, \emph{Nahm's conjecture: asymptotic computations
		and counterexamples}, Commun. Number Theory Phys. \textbf{5} (2011), no.~3,
	617--642. \MR{2864462}
	
	\bibitem[Zam91]{Zamolodchikov}
	A.~B. Zamolodchikov, \emph{On the thermodynamic {B}ethe ansatz equations for
		reflectionless {$ADE$} scattering theories}, Phys. Lett. B \textbf{253}
	(1991), no.~3-4, 391--394. \MR{1092210}
	
\end{thebibliography}
\end{document}